\documentclass[12pt, reqno]{amsart}
\usepackage{hyperref}

\usepackage{enumerate}
\usepackage{amssymb,amsmath,graphicx,amsthm, fullpage}
\usepackage{xcolor}
\usepackage{soul}

\DeclareMathOperator{\supp}{supp}

\usepackage{color}

\newcommand{\olive}[1]{\textcolor{olive}{#1}}

\def\sumJ(p-1){\underset{|J|=p-1}{{\sum}'}}

\def\sumKp-2{\underset{|K|=p-2}{{\sum}'}}

\def\sumjq+1{\underset {j\leq q+1}\sum}
\def\sumjn-1{\underset {j\leq n-1}\sum}

\def\bt{\begin{theorem}}
\def\el{\end{lemma}}
\def\bl{\begin{lemma}}
\def\et{\end{theorem}}
\def\bp{\begin{proposition}}
\def\ep{\end{proposition}}
\def\bd{\begin{definition}}
\def\ed{\end{definition}}
\def\br{\begin{remark}}
\def\er{\end{remark}}

\def\sumJ{\underset{|J|=q}{{\sum}'}}

\def\T{\text}
\newcommand{\Om}{\Omega}
\newcommand{\om}{\omega}

\newcommand{\bom}{\bar{\omega}}

\newcommand{\Dom}{\text{Dom} }

\newcommand{\we}{\wedge}

\newcommand{\no}[1]{\|{#1}\|}

\def\R{{\Bbb R}}

\def\I{{\mathcal I}}

\def\C{{\Bbb C}}

\def\la{\langle}
\def\ra{\rangle}
\def\di{\partial}
\def\dib{\bar\partial}

\DeclareMathOperator{\Lie}{Lie}

\numberwithin{equation}{section}

\def\T{\text}

\theoremstyle{plain}

\newtheorem{theorem}{Theorem}[section]

\newtheorem{corollary}[theorem]{Corollary}

\newtheorem{lemma}[theorem]{Lemma}

\newtheorem{proposition}[theorem]{Proposition}

\theoremstyle{definition}

\newtheorem{definition}[theorem]{Definition}

\theoremstyle{remark}

\newtheorem{remark}[theorem]{Remark}

\DeclareMathOperator{\Rre}{Re}

\newcommand{\p}{\partial}

\newcommand{\z}{\bar z}

\newcommand{\dbar}{\bar\partial}
\newcommand{\dbars}{\bar\partial^*}

\newcommand{\dbarb}{\bar\partial_b}

\newcommand{\vp}{\varphi}

\newcommand{\atopp}[2]{\genfrac{}{}{0pt}{2}{#1}{#2}}
\newcommand{\nn}{\nonumber}
\newcommand{\eps}{\epsilon}

\newcommand{\lam}{\lambda}

\newcommand{\opH}{\mathcal H}

\renewcommand{\H}{\mathcal H}
\def\Dom{\T{Dom}}
\newcommand{\opX}{\mathcal{X}}
\newcommand{\opY}{\mathcal{Y}}

\newcommand{\bdy}{\text{b}}
\newcommand{\opL}{\mathcal{L}}

\begin{document}

\title{Global regularity for the $\bar\partial$-Neumann problem on pseudoconvex manifolds}         
      \author{Tran Vu Khanh and Andrew Raich}          
               \address{T.V. Khanh}
\address{Department of Mathematics, International University, Vietnam National University Ho Chi Minh City, Vietnam}
\email{tkhanh@hcmiu.edu.vn}
        \address{A. Raich}
        \address{Department of Mathematical Sciences, SCEN 327, 1 University
        	of Arkansas, Fayetteville, AR 72701, USA} 
        \email{araich@uark.edu}
        \thanks{The first author was supported by a grant from the Vietnam National Foundation for Science and Technology Development (NAFOSTED) under grant number 101.02--2023.50. The second author  was supported by a grant from the Simons Foundation (707123, ASR)}

\begin{abstract}  
We establish general sufficient conditions for exact (and global) regularity in the $\dbar$-Neumann problem
on $(p,q)$-forms, $0 \leq p \leq n$ and $1\leq q \leq n$,
on a pseudoconvex domain $\Om$ with smooth boundary $b\Om$ in an $n$-dimensional complex manifold $M$. Our hypotheses include
two assumptions: 
\begin{enumerate}[i)]
	\item $M$ admits a function that is strictly plurisubharmonic acting on $(p_0,q_0)$-forms in a neighborhood of  $b\Om$ for some fixed
$0 \leq p_0 \leq n$, $1 \leq q_0 \leq n$,  or $M$ is a K\"ahler metric whose holomorphic bisectional curvature acting $(p,q)$-forms is positive; and 
\item  there exists a family of vector fields $T_\eps$ that are transverse to the boundary $\bdy\Om$ 
and generate one forms, which when applied to $(p,q)$-forms, $0 \leq p \leq n$ and $q_0 \leq q \leq n$, satisfy a ``weak form"
of the compactness estimate.
\end{enumerate}
 We also provide examples and applications of our main theorems.
\end{abstract} 

\maketitle

\section{Introduction}
In this paper, we extend and generalize the best known conditions for exact and global regularity for the $\dbar$-Neumann problem on domains
in $\C^n$ to domains $\Om$ in a complex manifold $M$. Finding conditions for the $\dbar$-Neumann operator $N_{p,q}$ to map 
$C^\infty_{p,q}(\bar\Om)$ to itself is one of the oldest and most important problems in the theory of $\dbar$. All known methods (including ours)
prove continuity on the $L^2$-Sobolev spaces $H^s_{p,q}(\Om)$, and this property is known as exact regularity.

The first global regularity result is due to Kohn and Nirenberg \cite{KoNi65} for domains in $\C^n$,
who proved that compactness of the $\dbar$-Neumann operator 
suffices. Catlin \cite{Cat84} established a general condition for establishing compactness of the
$\dbar$-Neumann operator, namely,  that there exist a family of bounded plurisubharmonic functions near $\bdy\Om$ 
with arbitrarily large complex Hessians.

In $\C^n$, more recent results have shown that compactness is not necessary to prove global regularity. In fact, 
if $\Om\subset \C^n$ admits a plurisubharmonic defining function or a certain family of vector fields that approximately
commutes approximately with $\dbar$, 
then the the $\dib$-Neumann operators are globally regular \cite{BoSt90, BoSt91}.  
Straube \cite{Str08} and Harrington \cite{Har11} each find ways to unify and generalize the two earlier approaches for global regularity.

Straube's global regularity theorem is the following.
\begin{theorem}[Straube \cite{Str08}]\label{thm:Straube}
	Let $\Om\subset\subset\C^n$  be a smooth pseudoconvex domain and $\rho$ be a defining
	function for  $\Om$. Let $0\le p\le n$ and $1 \le q \le n$. Assume that there is a constant $C$ such that for all $\epsilon > 0$  there exist a defining function $\rho_\epsilon$ for $\Om$
	and a constant $C_\epsilon$  with 
\begin{equation}\label{Straube-est1}
	C^{-1} \le | \nabla \rho_\epsilon| \le C
\end{equation}
  on $\bdy\Om$, and
\begin{eqnarray}\label{Straube-est2}
\left\|\sum_{L\in \I_{p}, K\in\I_{q-1}} \sum_{j,k=1}^n\frac{\di \rho}{\di z_j}  \frac{\di^2 \rho_\epsilon}{\di z_k \di \bar z_j} u_{L, kK} dz^L\we  d\bar z^K\right\|_0^2
\le \epsilon(\no{\dib u}_0^2+\no{\dib^* u}_0^2)+C_\epsilon\no{u}_{-1}^2
\end{eqnarray}
	for all $u \in  C_{p,q}^\infty(\bar\Om)\cap \T{Dom}(\dib^*)$.
	Then the $\dib$-Neumann operator $N_{p,q}$ acting on $(p, q)$-forms is exactly regular.
\end{theorem}
In the proof by Straube, the condition $\Om\subset \C^n$ is used heavily as the regularity of $N_{p,q}$ follows by the regularity of the weighted $\dib$-Neumann operator 
$N^t_{p,q}$ due to the Kohn weighted theory  \cite{Koh73}.\\

The first goal of this paper is to show that Straube's theorem generalizes to the general setting of complex manifolds. 

\begin{theorem}\label{thm:Khanh-Andy1} 
Let $M$ be a complex manifold and
	 $\Om\subset M$ be a smooth bounded pseudoconvex domain 
	which admits a strictly plurisubharmonic function acting on $(p_0,q_0)$-forms in a neighborhood of $\bdy\Om$. Let $\rho$ be a smooth defining function of $\Om$ and denote  $\gamma=\frac{1}{2}(\di \rho-\bar\di \rho)$.   Assume that there is a constant $C$ such that for all $\epsilon > 0$  there exist a purely imaginary vector field $T_\epsilon$ 
	and a constant $C_\epsilon$  with 
	\begin{eqnarray}\label{eqn:gamma Tepsilon}
	C^{-1} \le |\gamma(T_\epsilon)| \le C
	\end{eqnarray}
	 on $\bdy\Om$, and
\begin{eqnarray}\label{eqn:weak compactness 1}
\left\| \alpha_\epsilon\# u \right\|_0^2\le \epsilon(\no{\dib u}_0^2+\no{\dib^* u}_0^2)+C_\epsilon\no{u}_{-1}^2
\end{eqnarray}
for all $u \in  C_{(p_0,q_0)}^\infty(\bar\Om)\cap \T{Dom}(\dib^*)$, 
where $\alpha_\epsilon$ is the negative of the $(1,0)$-part of the real form $\Lie_{T_\epsilon}( \gamma)$.
Then, for $0\le p\le n$ and $q_0\le q\le n$, the space of $L^2$ harmonic $(p,q)$-forms $\H_{p,q}(\Om)\subset C^\infty_{p,q}(\bar\Om)$ and the operators $N_{p,q}$, $\dib N_{p,q}$, $\dib^*N_{p,q}$, $N_{p,q}\dib$, $N_{p,q}\dib^*$, $\dib\dib^*N_{p,q}$, $\dib^*\dib N_{p,q}$ $\dib^*N_{p,q}\dib$ and $\dib N_{p,q}\dib^*$  are exactly regular.
\end{theorem}
We define both the $\#$-operator and strictly plurisubharmonic functions on $(p_0,q_0)$-forms in 
Section \ref{sec:prelims}. It is worth noting that if $\Om$ is a bounded pseudoconvex in  $\C^n$, the space of $L^2$ harmonic $(p,q)$-forms is trivial for 
$q\ge 1$, i.e., $\H_{p,q}(\Om)=\{0\}$.\\

We now verify that in the case $M=\C^n$, the hypotheses of Theorem~\ref{thm:Straube} and  of Theorem \ref{thm:Khanh-Andy1} are equivalent. Let $\rho$ be a smooth defining function of $\Om$.  In a neighborhood $U$ of $b\Om$, we set 
$$\gamma = \frac{1}{2}(\di \rho -\dib \rho),\quad \text{and}\quad T=\frac{1}{|\di \rho |^2}\sum_{j=1}^n\left(\frac{\di\rho}{\di \z_j}\frac{\di}{\di z_j}-\frac{\di\rho}{\di z_j}\frac{\di }{\di \z_j}\right). $$
Then $\gamma(T)=1$. 
Assume the hypotheses of Theorem~\ref{thm:Straube} hold.  That means there exists $\rho_\epsilon$ is  the defining function of $\Om$  such that \eqref{Straube-est1} holds. Then there exist $h_\epsilon$ defined in a neighborhood $U$ of $b\Om$  such that $\rho_\epsilon = h_\epsilon \rho$ and   $C_1^{-1}\le |h_\epsilon|\le C_1$ in $U$ for some $C_1>0$ independent of $\epsilon$ (see  \cite[Lemma 2.5, page 51]{Ran86}). Let  $T_\epsilon=h_\epsilon  T$ in $U$. 
It follows $ \gamma(T_\epsilon)=h_\epsilon$ in $U$.
 This means \eqref{eqn:gamma Tepsilon} holds on $b\Om$. If the hypotheses of Theorem \ref{thm:Khanh-Andy1} hold, we define $\rho_\epsilon = \gamma(T_\epsilon) \rho$. It implies \eqref{eqn:gamma Tepsilon} holds. \\
 Now we show that \eqref{Straube-est2} and \eqref{eqn:weak compactness 1} are equivalent if $T_\epsilon = h_\epsilon T$ and $\rho_\epsilon=h_\epsilon \rho$ where  $h_\epsilon$ is a smooth function which is bounded away from zero on $b\Om$ uniformly in $\epsilon$.   Denote $\iota_{T_\epsilon}$ the contraction of forms with the vector field  $T_\epsilon$. Thus, 
 $$
 \iota_{T_\epsilon}(\gamma)=\gamma(T_\epsilon)=h_\epsilon\quad \text{and}\quad 
 \iota_{T_\epsilon}(d\gamma) =\olive{-}\frac{h_\epsilon}{|\di \rho|^2} \sum_{j,k=1}^n \left(\frac{\di \rho}{\di z_j}  \frac{\di ^2 \rho}{\di z_k \di\bar z_j}dz_k + \frac{\di \rho}{\di \bar z_j}  \frac{\di ^2 \rho}{\di z_j \di\bar z_k}d\bar z_k\right).
 $$
Using the Cartan formula and the fact that 
 $$  \frac{\di ^2 \rho_\epsilon}{\di z_k \di\bar z_j}= \frac{\di ^2 (h_\epsilon \rho )}{\di z_k \di\bar z_j}= h_\epsilon\frac{\di ^2 \rho}{\di z_k \di\bar z_j}+\frac{\di \rho}{\di z_k} \frac{\di  h_\epsilon }{\di \bar z_j }+\frac{\di  h_\epsilon }{\di z_k }\frac{\di \rho}{\di \bar z_j} +\rho \frac{\di^2  h_\epsilon }{ \di z_k\di\bar z_j},$$
 we obtain the expression of the $(1,0)$-form $\alpha_\epsilon$ of the 1-real form $ \Lie_{T_\epsilon}(\gamma)$ on $U$,
 \begin{align*}
\alpha_\epsilon =&\olive{-}\left( \Lie_{T_\epsilon}(\gamma) \right)^{1,0}\\
=&-\left( {\iota }_{T_\epsilon}(d\gamma)+ d(\iota_{T_\epsilon}(\gamma))\right)^{1,0}\\
 =&-\frac{h_\epsilon}{|\di \rho|^2} \sum_{j,k=1}^n\frac{\di \rho}{\di z_j}  \frac{\di ^2 \rho}{\di z_k \di\bar z_j}dz_k\olive{-}\di (h_\epsilon)\\
 =&-\frac{1}{|\di \rho|^2} \sum_{j,k=1}^n\left(\frac{\di \rho}{\di z_j}  \frac{\di ^2 \rho_\epsilon}{\di z_k \di\bar z_j} -\frac{\di \rho}{\di z_k}  \frac{\di \rho}{\di z_j}\frac{\di  h_\epsilon }{\di \bar z_j }   -\left|\frac{\di \rho}{\di z_j}\right|^2 \frac{\di  h_\epsilon }{\di z_k }   -\rho \frac{\di \rho}{\di z_j} \frac{\di^2  h_\epsilon }{ \di z_k\di\bar z_j}\right)dz_k -\p h_\eps.\\
 =&-\frac{1}{|\di \rho|^2} \sum_{j,k=1}^n\left(\frac{\di \rho}{\di z_j}  \frac{\di ^2 \rho_\epsilon}{\di z_k \di\bar z_j} -\frac{\di \rho}{\di z_k}  \frac{\di \rho}{\di z_j}\frac{\di  h_\epsilon }{\di \bar z_j }     -\rho \frac{\di \rho}{\di z_j} \frac{\di^2  h_\epsilon }{ \di z_k\di\bar z_j}\right)dz_k.
\end{align*}
 Thus, for any $u\in C_{p,q}^\infty(\bar\Om)\cap \T{Dom}(\dib^*)$, we observe that
\[
\alpha_\epsilon \# u =  \frac{1}{|\di \rho|^2} \sum_{j,k=1}^n\frac{\di \rho}{\di z_j}  \frac{\di ^2 \rho_\epsilon}{\di z_k \di\bar z_j}dz_k
\]  
since the remaining terms vanish on the boundary $b\Om$. Indeed, 
 $$\left(\sum_{j,k=1}^n \frac{\di \rho}{\di z_k}  \frac{\di \rho}{\di z_j}\frac{\di  h_\epsilon }{\di \bar z_j } dz_k  \right)\# u = \left(\sum_{j=1}^n \frac{\di \rho}{\di z_j}\frac{\di  h_\epsilon }{\di \bar z_j } \right)\left(\sum_{L\in \I_{p}, K\in\I_{q-1}} \sum_{k=1}^n  \frac{\di \rho}{\di z_k}  u_{L,kK}dz^L\we d\bar z^K\right) =0$$
since $u\in \T{Dom}(\dib^*)$. Therefore, \eqref{Straube-est2} and \eqref{eqn:weak compactness 1} are equivalent by using the elliptic estimates of the $\dib$-Neumann problem on compactly supported forms. \\

We remark that Theorem \ref{thm:Khanh-Andy1} is a proper generalization of Theorem \ref{thm:Straube}. For example, the case when $M$ is a Stein manifold is not covered by Theorem \ref{thm:Straube}. \\

The second goal of this paper is to relax the ``the existence of a strictly plurisubharmonic function" condition, at the expense of requiring stronger estimate than \eqref{eqn:weak compactness 1}.  
\begin{theorem}\label{thm:Khanh-Andy2}
Let $M$ be a complex manifold and $\Om\subset \subset M$ be a smooth bounded pseudoconvex domain.	 
	 Assume that there are constants $C,c$ such that for all $\epsilon > 0$  there exist a purely imaginary vector field $T_\epsilon$ 
	and a constant $C_\epsilon$  with 
	$$C^{-1} \le |\gamma(T_\epsilon)| \le C$$
	on $\bdy\Om$, and
	\begin{eqnarray}\label{eqn:main2}
	\no{u}_0^2+\frac{1}{\epsilon}\left(\no{\bar\alpha_\epsilon\we u}_0^2
	+\no{\alpha_\epsilon\# u}_0^2\right)\le c(\no{\dib u}_0^2+\no{\dib^* u}_0^2)+C_\epsilon\no{u}_{-1}^2
	\end{eqnarray}
	for all $u \in  C_{p,q}^\infty(\bar\Om)\cap \T{Dom}(\dib^*)$, where $\alpha_\epsilon$ is defined in Theorem \ref{thm:Khanh-Andy1}.
	Then the space of $L^2$ harmonic $(p,q)$-forms $\H_{p,q}(\Om)\subset C^\infty_{p,q}(\bar\Om)$ and the operators $N_{p,q}$, $\dib N_{p,q}$, $\dib^*N_{p,q}$, $N_{p,q}\dib$, $N_{p,q}\dib^*$, $\dib\dib^*N_{p,q}$, $\dib^*\dib N_{p,q}$ $\dib^*N_{p,q}\dib$ and $\dib N_{p,q}\dib^*$   are exactly regular.
\end{theorem}

\begin{remark}In Theorem~\ref{thm:Khanh-Andy2}, we are unable to extend the degrees of forms since there is no information about the $L^2$ basic estimate in other degrees. 
 	\end{remark}
As a corollary of Theorem~\ref{thm:Khanh-Andy2},  
we can establish a general global regularity for the $\dib$-Neumann problem on pseudoconvex domains in 
K\"ahler manifolds with positive holomorphic bisectional curvatures. 
\begin{theorem}\label{thm:Khanh-Andy3} Let $M$ be a K\"ahler manifold whose holomorphic bisectional curvature acting $(p,q)$-forms is positive.  Let $\Om\subset\subset M$ be  a smoothly pseudoconvex domain which admits a plurisubharmonic defining function for $\Om$. Then there is a constant $C$ such that for all $\epsilon > 0$  there exists a purely imaginary vector field $T_\epsilon$  with 
	$$C^{-1} \le |\gamma(T_\epsilon)| \le C\quad \T{and}\quad |\alpha_\epsilon|\le \epsilon$$
	on $\bdy\Om$. Moreover, the space of $L^2$ harmonic $(p,q)$-forms $\H_{p,q}(\Om)$ is trivial, i.e., $\H_{p,q}(\Om)=\{0\}$ and the operators $N_{p,q}$, $\dib N_{p,q}$, $\dib^*N_{p,q}$, $N_{p,q}\dib$, $N_{p,q}\dib^*$, $\dib\dib^*N_{p,q}$, $\dib^*\dib N_{p,q}$ $\dib^*N_{p,q}\dib$ and $\dib N_{p,q}\dib^*$   are exactly regular.
	\end{theorem}

The holomorphic bisectional curvature of the Fubini-Study metric of $\C\mathbb P^n$ acting on $(0,q)$-forms is positive when
with $q\ge 1$, which immediately yields following corollary.
\begin{corollary} Let $\Om\subset\subset \C\mathbb P^n$ be  a smooth pseudoconvex domain which admits a plurisubharmonic defining function for $\Om$. 
Then the conclusion of Theorem ~\ref{thm:Khanh-Andy3} holds on $\Om$ for  $p=0$  and $1\le q\le n$. 
\end{corollary}

The outline of the rest of the paper is as follows. The technical preliminaries are
given in Section~\ref{sec:prelims}. Controlling derivatives with $\dbar$ and $\dbars$ are established  in
Section~\ref{sec:derivatives}.  We prove Theorem~\ref{thm:Khanh-Andy1} in Section~\ref{sec:proof of thm1}. Its proof follows
the argument of \cite{Str08} and \cite {KhRa20} in a general setting. The proofs of Theorem~\ref{thm:Khanh-Andy2} and Theorem~\ref{thm:Khanh-Andy3} are given in Section~\ref{sec:Proofs}. In the last section, we introduce a new version of elliptic regularization.
%
%

\section{Preliminaries}\label{sec:prelims}
\subsection{Complex geometry}

Our setup follows \cite{Hor65,Koh73}
Let $M$ be a paracompact $n$-dimensional complex manifold and $\Om\subset M$ a smooth, open submanifold with compact closure. 
The boundary of $\Om$ is denoted by $\bdy\Om$. 
On $M$, fix a smooth Hermitian metric $g$ so that in appropriate local coordinates,
\[
g = \sum_{j,k=1}^n g_{j\bar k}\, dz_j \otimes d\z_k
\]
and its associated Hermitian form
\begin{equation}\label{eqn:metric defn}
\om = i \sum_{j,k=1}^n g_{j\bar k}\, dz_j \wedge d\z_k.
\end{equation}
As usual, $g^{j\bar k}$ will denote the inverse matrix to $g_{j\bar k}$ and the induced metric on the cotangent space (still called $g$) is
\[
g = \sum_{j,k=1}^n g^{j\bar k}\, \frac{\p}{\p z_j} \otimes \frac{\p}{\p\z_k}.
\]
We denote the space of increasing $q$-tuples by $\I_q$, that is,
$$\I_q = \{ J = (j_1,\dots,j_q) : 1 \leq j_1 < j_2 < \cdots < j_q \leq n\}.$$
Let $T^{p,q}(M)$ be the tangent bundle of $(p,q)$-vectors (with smooth coefficients) and
$\Lambda^{p,q}(M)$ the space of $(p,q)$-forms on $M$. 
In an analytic coordinate system $z_1,\dots,z_n$ and $\la\cdot,\cdot\ra$ is the inner product induced by $g$, then
for $u,v\in\Lambda^{p,q}$,
\[
u = \sum_{\atopp{I\in\I_p}{J\in\I_q}} u_{IJ}\, dz^I\wedge d\z^J, \quad 
v = \sum_{\atopp{I\in\I_p}{J\in\I_q}} v_{IJ}\, dz^I\wedge d\z^J
\quad\text{and}\quad
\la u,v\ra = \sum_{\atopp{I,K\in\I_q}{J,L\in\I_q}} u_{IJ}\overline{v_{KL}} g^{IJ,KL}
\]
where
\[
g^{IJ,KL} = \la dz^I\wedge d\z^J, dz^K\wedge d\z^L \ra.
\]
A defining function $\rho$ for $\Om$ is a $C^\infty$ function defined on a neighborhood
of $\bar\Om$ so that $\Om = \{\rho<0\}$ and $|d\rho|=1$ on $\bdy\Om$. 
The coordinate functions $z_1,\dots,z_n$ are not orthogonal, and we can Gram-Schmidt to obtain $(1,0)$
forms $\om_1,\dots,\om_n$ so that 
\[
\la\om_j,\om_k\ra = \delta_{jk}
\]
and any $(p,q)$-forms $u,v$ satisfy
\[
u = \sum_{\atopp{I\in\I_p}{J\in\I_q}} u_{IJ}\, \om^I\wedge\bom^J, \quad 
v = \sum_{\atopp{I\in\I_p}{J\in\I_q}} v_{IJ}\, \om^I\wedge\bom^J
\quad\text{and}\quad
\la u, v\ra = \sum_{\atopp{I\in\I_p}{J\in\I_q}} u_{IJ} \overline{v_{IJ}}.
\]
Given a function $\lambda$ that is smooth and bounded near $\bdy\Om$, we say that $\lambda$ is
\emph{strictly plurisubharmonic on $(p_0,q_0)$-forms} if given any $(p_0,q_0)$-form $u$ and a local coordinate patch $U$ near $\bdy\Om$,
\[
\sum_{\atopp{I\in\I_{p_0}}{K\in\I_{q_0-1}}} \sum_{j,k=1}^n \frac{\p^2\lambda}{\p z_j \p\z_k} u_{I,jK} \overline{u_{I,kK}} >0.
\]

\subsection{Sobolev Spaces on $M$}
To express global objects, we need a locally finite partition of unity $\{\eta_\alpha\}$ that is both subordinate to the cover $\{U_\alpha\}$
and sufficiently refined so that $\om$ can be expressed as in \eqref{eqn:metric defn}. Let 
$T^{p,q}(\Om)$ be the bundle of smooth $(p,q)$-vectors on $\Om$
and $\Lambda^{p,q}(\Om)$ be the bundle of smooth $(p,q)$-forms on $T^{p,q}(\Om)$. 
Locally, this means $u\in\Lambda^{p,q}(\Om)$ can be expressed
\[
u(z) = \sum_{I\in\I_p}\sum_{J\in\I_q} u_{IJ} \, dz^I \wedge d\z^J.
\] 
We define $L^2_{p,q}(\Om)$ as the completion of $\Lambda^{p,q}(\Om)$ under the inner product
\[
(u,v)_{L^2(\Om)} = \int_\Om \la u,v \ra \, dV
\]
for $u, v \in \Lambda^{p,q}(\Om)$ and volume measure $dV$. If $\phi$ is a function defined on a neighborhood of $\bar\Om$, then
the weighted $L^2$ space, $L^2_{p,q}(\Om,\phi)$ has inner product
\[
(u,v)_{L^2(\Om,\phi)} = \int_\Om \la u,v \ra  e^{-\phi}\, dV
\]
and norm $\|u\|_\phi^2 = (u,u)_{L^2(\Om,\phi)} = (u,u)_\phi$.

Next, for each
positive integer $s$, we define the Sobolev space $H^s(\Om)$ as the completion of $C^\infty(\Om)$ under the inner product
\[
(u,v)_{H^s} 
= \sum_{\alpha\in \mathfrak I} \sum_{|\gamma| \leq s} 
\big(D^{\gamma} (\sqrt{\eta_\alpha} u),
D^{\gamma} (\sqrt{\eta_\alpha} v) \big)_{H^s}
\]
where the $|\gamma|$-th order derivative $D^\gamma$ is taken via a local (real) unitary frame $\{X_1, \dots, X_n, Y_1, \dots, Y_n\}$.
By including $\sqrt{\eta_\alpha}$ in the definition of $H^s_{p,q}(\Om)$, it is immediate that the $H^0_{p,q}(\Om)$ and $L^2_{p,q}(\Om)$-norms are equal.

Given a vector field $X$ and a form $u$, let $Xu$ denote
differentiation of the components of $u$ by $X$ in a given coordinate system. In particular, if $u$ is a $(p,q)$-form given locally by
$u = \sum_{I\in\I_q}\sum_{J\in\I_q} u_{IJ}\, dz^I\wedge d\z^J$, then if
\[
Xu = \sum_{I\in\I_q}\sum_{J\in\I_q} Xu_{IJ}\, dz^I\wedge d\z^J.
\]
The operator $u \mapsto Xu$ depends on our choice of coordinates, but we will always be using the good coordinate charts that we
defined when constructing the Sobolev spaces. Since we are concerned with regularity in Sobolev spaces, we have already made choices
that depend on the metric and coordinate charts.

We define the \emph{metric contraction operator} $\#:\Lambda^{1,0}\times\Lambda^{p,q} \to \Lambda^{p,q-1}$ by
\[
\la \bar\theta\wedge u, v \ra = \la u, \theta \# v \ra
\]
for all $\theta \in \Lambda^{1,0}$, $u \in \Lambda^{p,q-1}$, and $v \in \Lambda^{p,q}$. In the case that $p=0$ and we 
are have orthonormal forms $\om_1,\dots,\om_n$ on a neighborhood $U$, then 
\[
\theta\# u=\sum_{I\in\I_{q-1}} \left(\sum_{j=1}^n \theta_j u_{jI}\right)\bom^I\quad\T{if $\theta=\sum_j\theta_j\,\om_j$ is a $(1,0)$-form on $U$}.
\]
We also define a metric contraction operator $\#:\Lambda^{1,1}\times\Lambda^{p,q} \to \Lambda^{p,q}$ by
\[
\quad\theta\# v=\sum_{I\in\I_{q-1}}\sum_{k=1}^n \left(\sum_{j=1}^n \theta_{jk} v_{jI}\right)\bom_k\we\bom^{I}
\quad\T{if $\theta=\sum_{j,k=1}^n \theta_{jk}\,\om_j\we\bom_k$ is a $(1,1)$-form on $U$}.
\]
Note that
\[
\theta \# v = \sum_{j,k=1}^n \theta_{jk}\,\bom_k \wedge\big(\om_j\#v\big).
\]
We will refer to both of these operators as the $\#$-operator. The $\#$-operator depends on the choice of metric but not the choice
of coordinates.

We will use the small constant/large constant inequality, namely, for any $\delta>0$ and positive numbers $a,b>0$
\begin{equation}\label{eqn:sc-lc} \tag{sc-lc}
ab \leq \frac \delta2 a^2 + \frac{1}{2\delta} b^2.
\end{equation}

\subsection{The Levi form on $\Om$}
The operator 
\[
\dbar_{p,q}  = \dbar : L^2_{p,q}(\Om)\to L^2_{p,q+1}(\Om)
\] 
is defined with its $L^2$-maximal definition and the adjoint
\[
\dbars_{p,q}  = \dbars :L^2_{p,q+1}(\Om) \to L^2_{p,q}(\Om)
\]
is defined with respect to the inner product $(\cdot,\cdot)$. The $\dbar$-Neumann Laplacian
$\Box_{p,q} = \dbar_{q-1}\dbars_{q-1}+\dbars_q\dbar_q$, 
and the space of harmonic forms
\[
\opH_{p,q}(\Om) = L^2_{p,q}(\Om) \cap \ker \dbar \cap \ker\dbars.
\]
When it exists, the $\dbar$-Neumann operator
\[
N_{p,q}: L^2_{p,q}(\Om) \to \Dom(\Box_{p,q})
\]
satsifies $\Box_{p,q} N_{p,q} = I - H_{p,q}$ where $H_{p,q} :  L^2_{p,q}(\Om) \to \opH_{p,q}(\Om)$ is the orthogonal projection.

We denote the Bergman projection by $P_{p,q}$. 
Also, when the $\dbar$-Neumann operator
exists,  for any $(p,q)$-form $u$, there is the Hodge decomposition
\[
u = \dbar\dbars N_{p,q}u + \dbars \dbar N_{p,q} u + H_{p,q}u
\]
which means the harmonic projection
\[
H_{p,q} = I - \dbar\dbars N_{p,q} - \dbars\dbar N_{p,q}.
\]

For any smooth defining function $\rho$, the \emph{Levi form} $\opL_\rho$ is the real element of $\Lambda^{1,1}(\bdy\Om)$ 
defined by
\[
\opL_\rho(i\bar L \wedge L') = i\p\dbar\rho(i \bar L \wedge L') = \partial\rho([L,\bar L'])
\]
where the second equality follows from the Cartan identity. As usual, if $\tilde\rho$ is another smooth defining function, then
$\tilde \rho = h \rho$ for some nonvanishing smooth function $h$ and $\opL_{\tilde \rho} = h \opL_\rho$. We typically suppress
the subscript $\rho$ whenever the choice of defining function is not relevant.

A domain $\Om\subset M$ is called \emph{pseudoconvex} if the Levi form is semi-definite on $\bdy\Om$. In the case that
$\opL$ is negative semi-definite, we simply reverse the orientation of $M$ and we may therefore assume that $\opL$ is
positive semi-definite on $\bdy\Om$.  

%
%
\section{Controlling Derivatives with $\dbar$ and $\dbars$}\label{sec:derivatives}

The next lemma establishes estimates for ``benign derivatives", see, e.g., \cite[Lemma 5.6]{Str10}.
\begin{lemma} \label{lem:controlling Xu, Yu}
	Let $M$ be a complex manifold and $\Om\subset\subset  M$ be a smooth, pseudoconvex domain. Then for any 
	$s\in \mathbb N$ and $u\in C^\infty_{p,q}(\overline\Om)\cap \T{Dom}(\dib^*)$: \\
	(i) If $X$  be an $(1,0)$ vector field, smooth on $\bar\Om$ then
	\begin{equation}\label{eqn:bar X}
		\no{\bar Xu}^2_{s-1}\le c_s\left(\no{\dib u}^2_{s-1}+\no{\dib^* u}^2_{s-1}+\no{ u}^2_{s-1}\right).
	\end{equation}
	(ii) If $Y$  be an $(1,0)$ vector field, smooth on $\bar\Om$ with $Y\rho=0$ on $\bdy\Om$ then
	\begin{equation}\label{eqn:Y}
		\no{Yu}^2_{s-1}\le c_s\left(\no{\dib u}^2_{s-1}+\no{\dib^* u}^2_{s-1}+\no{ u}_{s-1}\no{u}_{s}\right).
	\end{equation} 	
		(iii) If $Z$  be a  vector field, smooth on $\bar\Om$ with $Z|_{\bdy\Om}=0$  then
	\begin{equation}\label{eqn:Z}
	\no{Zu}^2_{s-1}\le c_s\left(\no{\dib u}^2_{s-1}+\no{\dib^* u}^2_{s-1}+\no{ u}_{s-1}^2\right).
	\end{equation} 	
\end{lemma}

\begin{proof}
The proofs of (i) and (ii) are standard, and the proofs for domains in $\C^n$ apply here as well. See, e.g., \cite{BoSt91,Str10,Har11}.
To prove (iii), we observe that it is equivalent to prove 
	\[
	\|v\|_s^2 \leq c_s \big(\|\dbar v\|_{s-1}^2 + \|\dbars v\|_{s-1}^2 + \|v\|_{s-1}^2\big)
	\]
	for $v \in (C^\infty_{p,q})_0(\bar\Om)$, that is, $(p,q)$-forms $v$ that vanish at the boundary. Moreover, as normal derivatives can be written in terms of coefficients
	of $\dbar$, $\dbars$, tangential derivatives, and the coefficients of the form itself (\cite[Lemma 2.2]{Str10} and as this result is local, it applies in the more general complex manifold setting), we may assume
	that all differential operators in this proof are tangential at the boundary. Additionally, if we have an $s$-order operator $D^\gamma Y$ where
	$|\gamma|=s-1$, we may further assume that $Y$ is a type $(1,0)$-vector field, since we already established the 
	desired result for $(0,1)$-vector fields is part (i). 
	Consequently,
	\begin{align*}
	(D^\gamma Y v, D^\gamma Y v) 
	&\leq (D^\gamma Y^*v, D^\gamma Y^*v) + C_s \|v\|_{s-1}\|Y v\|_{s-1} + C_s \|v\|_{s-1}\|Y^* v\|_{s-1} + (D^\gamma v, D^\gamma[Y^*,Y] v) \\
	&\leq \|Y^*v\|_{s-1}^2 + C_s \|v\|_{s-1}\|Y v\|_{s-1} + C_s \|v\|_{s-1}\|Y^* v\|_{s-1} + \|v\|_{s-1}\|\|v\|_s
	\end{align*}
	The bounds for the first term follow from (i) and the remaining terms from a \eqref{eqn:sc-lc} argument (and a reabsorption of
	$\|Y' v\|_{s-1}$ and $\|v\|_s$).
\end{proof}

\begin{lemma}\label{lmm:Hs Ts}
	Let $T_\epsilon$ be a purely imaginary vector field, smooth on $\bar\Om$ with $C^{-1}\le |\gamma(T_\epsilon)|<C$ on $\bdy\Om$. If $s\geq 1$ and $u \in H^s_{p,q}(\Om)\cap\Dom(\dbars)$, 
	then
	\begin{equation}\label{eqn: T esilon 1}
		\no{u}^2_{s}\le c_{\epsilon,s}\left(\no{\dib u}^2_{s-1}+\no{\dib^* u}^2_{s-1}+\no{ u}^2_{s-1}\right)+c_s\no{T_\epsilon^s u}_0^2,
	\end{equation}
	and if $s \geq 2$ and $u\in H^s_{p,q}(\Om)\cap\Dom(\Box)$, then
	\begin{equation}\label{eqn:T esilon 2}
		\no{u}^2_{s}\le c_{\epsilon,s}\left(\no{\Box u}^2_{s-2}+\no{ u}^2_{s-1}\right)+c_s\no{T_\epsilon^s u}_0^2.
	\end{equation} 	
\end{lemma}
\begin{proof}Elliptic estimates for $\dbar$ fail only near $\bdy\Om$, so it suffices, as in Lemma \ref{lem:controlling Xu, Yu}, to consider
	smooth $u$ with $\supp u$ in a special boundary neighborhood $U$ with boundary chart $\{X_1,\dots,X_n,Y_1,\dots,Y_n\}$ so that
	the real normal to $\bdy\Om\cap U$ is $X_n$, while $iY_n = T$, and the complex normal is $\frac 12(X_n - i Y_n)$. Additionally, by density, we may assume that $u \in H^{s+1}_{p,q}(\Om)\cap\Dom(\dbars)$.
	If $D$ is the gradient operator, then 
	Lemma \ref{lem:controlling Xu, Yu} implies
	\begin{align*}
		\|u\|_s^2 
		&\leq c \|D u\|_{s-1}^2
		\leq c\Big( \sum_{j=1}^{n-1}\big( \|X_j u\|_{s-1}^2 + \|Y_j u\|_{s-1}^2\big) + \|X_n u\|_{s-1}^2 + \|Y_n u\|_{s-1}^2 \Big)\\
		&\leq c_s \sum_{j=1}^{n-1}\big( \|X_j u\|_{s-1}^2 + \|Y_j u\|_{s-1}^2\big) + c_s \|\dbar u\|_{s-1}^2 + c_s\|\dbars u\|_{s-1}^2 + c\|T_\eps u\|_{s-1}^2 
		+ c_{s} \|u\|_{s} \|u\|_{s-1}
	\end{align*}
A \eqref{eqn:sc-lc} argument and an absorbtion of $\|u\|_s^2$ by the LHS, we estimate
\[
\|u\|_s^2  \leq c_s\left(\no{\dib u}^2_{s-1}+\no{\dib^* u}^2_{s-1}+ c_{\eps,s}\no{ u}^2_{s-1}\right) + c\|T_\eps u\|_{s-1}^2.
\]

	By a simple induction argument (essentially repeating the argument of the preceeding paragraph), we may bound
	\[
	\|T_\eps u\|_{s-1}^2 \leq c_s\big(\|T_\eps^s u\|_0^2 + \|\dbar u\|_{s-1}^2 + \|\dbars u\|_{s-1}^2\big) + c_{s,\eps}\|u\|_{s-1}^2
	\]
	which establishes \eqref{eqn: T esilon 1}.
	
	Next, applying \eqref{eqn: T esilon 1} to $\|\dbar u\|_{s-1}^2$ and $\|\dbars u\|_{s-1}^2$, we see
	\[
	\|\dbar u\|_{s-1}^2 \leq c_{\eps,s}\big(\|\dbars\dbar u\|_{s-2}^2 + \|\dbar u\|_{s-2}^2\big) + \|T^{s-1}_\eps \dbar u\|_0^2
	\]
	and
	\[
	\|\dbars u\|_{s-1}^2 \leq c_{\eps,s}\big(\|\dbar\dbars u\|_{s-2}^2 + \|\dbars u\|_{s-2}^2\big) + \|T^{s-1}_\eps \dbars u\|_0^2.
	\]
	Since $u\in\Dom(\Box)$, we compute (if $s\geq 2$)
	\begin{align*}
		\|T^{s-1}_\eps \dbar u\|_0^2
		&= \big(T^{s-1}_\eps \dbars\dbar u, T^{s-1}_\eps u \big) + \big([\dbars,T^{s-1}_\eps]\dbar u,T^{s-1}_\eps u\big) + \big(T^{s-1}\dbar u,[T^{s-1}_\eps,\dbar] u\big) \\
		&= \big(T^{s-2}_\eps \dbars\dbar u, T^s_\eps u \big) + \big([\dbars,T^{s-1}_\eps]\dbar u,T^{s-1}_\eps u\big) + \big(T^{s-1}\dbar u,[T^{s-1}_\eps,\dbar] u\big) \\
		&\leq c_{\eps, s}\|\dbars\dbar u\|_{s-2}\|T^s_\eps u\|_0 + c_{\eps,s} \|\dbar u\|_{s-1}\| \|u\|_{s-1}.
	\end{align*}
	A \eqref{eqn:sc-lc} argument allows for the absorption of $ \|\dbar u\|_{s-1}$ and the following estimation of $\|\dbars\dbar u\|_{s-2}$.
	The estimate for $\|T^{s-1}_\eps \dbars u\|_0^2$ is identical, leaving us to estimate
	$\|\dbars\dbar u\|_{s-2}^2 + \|\dbar\dbars u\|_{s-2}^2$. Let $X_s$ be a derivative of order $s-2$. Then since $u\in\Dom(\Box)$,

\begin{align*}
&\big( X_s \dbars\dbar u, X_s \dbars\dbar u\big) + \big( X_s \dbar\dbars u, X_s \dbar\dbars u\big) \\
&= \Rre\Big\{\big( X_s \Box u, X_s \dbars\dbar u\big) + \big( X_s \Box u, X_s \dbar\dbars u\big) 
-\big( X_s \dbar\dbars u, X_s \dbars\dbar u\big)  -\big( X_s \dbars\dbar u, X_s \dbar\dbars u\big) \Big\} \\
&=\big( X_s \Box u, X_s \Box u\big) -2\Rre\Big\{ \big( X_s \dbar\dbars u, X_s \dbars\dbar u\big)  \Big\} \\
&= \|X_s \Box u\|_0^2 -2\Rre\Big\{ \big( X_s \dbar\dbars u, [X_s, \dbars]\dbar u\big)  + \big( [\dbar,X_s] \dbar\dbars u, X_s\dbar u\big) \Big\} \\
&\leq\|X_s \Box u\|_0^2 + C\big(\|X_s \dbar\dbars u\|_0 \|u\|_{s-1} + \|\dbar\dbars u\|_{s-2}\|u\|_{s-1}\big).
\end{align*}

The $H^{s-2}$-norm is built from derivatives of the form $X_s$. This fact and a \eqref{eqn:sc-lc} argument suffices to prove
\eqref{eqn:T esilon 2}.
\end{proof}

\begin{lemma}\label{lmm:commutators with Ts}
Let $T_\epsilon$ be a purely imaginary vector field, smooth on $\bar\Om$ with $C^{-1}\le |\gamma(T_\epsilon)|<C$ on $\bdy\Om$.
Then 
	\begin{equation}\label{eqn: T epsilon s}T_\eps^s=h_\eps^s T^s+\mathcal Z_{\eps,s}
		\end{equation}
	\begin{equation}\label{eqn: T epsilon wedge}
[\dib, T^s_\epsilon]=-s h_\epsilon^{-1} \bar\alpha_\epsilon\we T_\epsilon^s+\mathcal X_{\epsilon,s}
	\end{equation}
	and
\begin{equation}\label{eqn:T esilon conrner}
[\dib^*, T^s_\epsilon]=s h_\epsilon^{-1} \alpha_\eps \# T_\epsilon^s+\mathcal Y_{\epsilon,s}
\end{equation} 	
where $h_\epsilon=\gamma(T_\epsilon)\in \R$ on $\bdy\Om$, 
and $\mathcal X_{\epsilon,s}:C^\infty_{p,q}(\bar\Om)\to C^\infty_{p,q+1}(\bar\Om)$, 
$\mathcal Y_{\epsilon,s}:C^\infty_{p,q}(\bar\Om)\to C^\infty_{p,q-1}(\bar\Om)$, and 
$\mathcal Z_{\epsilon,s}:C^\infty_{p,q}(\bar\Om)\to C^\infty_{p,q}(\bar\Om)$ are operators satisfying 
\begin{equation}\label{eqn:X,Y,Z ests}
\|\mathcal X_{\epsilon,s}u\|_0^2+ \|\mathcal Y_{\epsilon,s}u\|_0^2
+\|\mathcal Z_{\epsilon,s}u\|_0^2\le c_{\epsilon,s}\left(\no{\dib u}^2_{s-1}+\no{\dib^* u}^2_{s-1}+\no{u}^2_{s-1}+\no{u}_{s-1}\no{u}_{s}\right) 
\end{equation}
for $u \in C^\infty_{p,q}(\Om)\cap\Dom(\dbars)$.
\end{lemma}

\begin{remark} We comment that \eqref{eqn:T esilon conrner}  means that if
$u$ is a $(p,q)$-form, then \eqref{eqn:T esilon conrner} means that on a good, local chart
\[
s h_\epsilon^{-1} \alpha_\eps\# T_\epsilon^s u
=s h_\epsilon^{-1}\alpha_\eps\# \Big(\sum_{\atopp{I\in\I_p}{J\in\I_q}} T_\epsilon^s u_{IJ}\, \om^I\we \bom^J\Big).
\]
\end{remark}

\begin{proof}[Proof of Lemma~\ref{lmm:commutators with Ts}]  We only need to prove \eqref{eqn: T epsilon s}, \eqref{eqn: T epsilon wedge} and 
\eqref{eqn:T esilon conrner} for the case $s=1$ since the higher degrees follow by induction. For example,
\[
[\dbar,T^2_\eps] = 2[\dbar,T_\eps]T_\eps + \big[T_\eps,[\dbar,T_\eps]\big].
\]
The uniformly bounded condition of $\gamma(T_\eps)$ on $\bdy\Om$ implies that 
\[
T_\eps=h_\eps T+ \mathcal Z_\eps\quad \T{on $\bar\Om$, where } \mathcal Z_\eps=\bar X_\eps+ Y_\eps+Z_\eps.
\]
Here, $X_\eps, Y_\eps, Z_\eps$ are smooth $(1,0)$-vector fields as in Lemma \ref{lem:controlling Xu, Yu}.

By Lemma~\ref{lem:controlling Xu, Yu}, 
\[
\|\mathcal Z_\eps u\|_0^2
\le c_\eps \left(\|\dib u\|_0^2+\|\dib^* u\|_0^2+\|u\|_0^2+\|u\|_0\|u\|_1\right),
\]
for any $u\in C^\infty_{p,q}(\bar\Om)\cap \Dom(\dib^*)$, and hence \eqref{eqn: T epsilon s} for $s=1$ follows.\\
	 
For the proof of \eqref{eqn: T epsilon wedge}, we first recall that   $\alpha_\eps$ is the $(1,0)$-component of the form $-\Lie_{T_\eps}(\gamma)$ and 
$\gamma=\frac{1}{2}(\di \rho-\dib \rho)$. 
Working on a local patch of $M$, we let $\alpha_{j,\eps}$ be the $\om_j$-component of $\alpha_\eps$. 
Then
\[
\alpha_{j,\eps}
=\alpha_\eps(L_j)=-\Lie_{T_\eps}(\gamma)(L_j)
=-\big(T_\eps \gamma(L_j) -\gamma \{[T_\eps,L_j]\}\big)
=\gamma \{[T_\eps,L_j]\}.
\] 
This implies 
\[
[T_\eps,L_j]
= \alpha_{j,\eps}T+\tilde{\mathcal Z}_{j,\eps}
= h_\eps^{-1} \alpha_{j,\eps}T_\eps+\mathcal Z_{j,\eps}
\]
where $\mathcal Z_{j,\eps}$ is a vector field satisfying 
\[
\|\mathcal Z_{j,\eps} u\|_0^2\le c_\eps \big(\|\dib u\|_0^2+\|\dib^* u\|_0^2+\|u\|_0^2+\|u\|_0\|u\|_1\big).
\]

Therefore, since $T_\eps$ is purely imaginary, 
\begin{align*}
[\dib,T_\eps]u
&= \sum_{j=1}^n\big( \big[\bar L_j,T_\eps\big] \, \bom_j\big)\we u + \mathcal{X}_{\eps,0} u\\ 
&=-h^{-1}_\eps\sum_{j=1}^n \bar\alpha_{j,\eps} \, \bom_j \we T_\eps u
+ \sum_{j=1}^n \bom_j\we  \overline{\mathcal Z_{j,\eps}}  u+\mathcal{X}_{\eps,0} u\\
&=-h^{-1}_\eps \bar \alpha_\eps\we T_\eps u +\mathcal X_{\eps,1}u
\end{align*}
and similarly, by taking adjoints,
\[
[\dib^*,T_\eps]u = h^{-1}_\eps \alpha_\eps\# T_\eps u +\mathcal Y_{\eps,1}u;
\]
where $\mathcal X_{\epsilon,1}:C^\infty_{p,q}(\bar\Om)\to C^\infty_{p,q+1}(\bar\Om)$, 
$\mathcal Y_{\epsilon,1}:C^\infty_{p,q}(\bar\Om)\to C^\infty_{p,q-1}(\bar\Om)$ satisfy
\[
\|\mathcal X_{\epsilon,1}u\|_0^2+\|\mathcal Y_{\epsilon,1}u\|_0^2\le c_{\epsilon} \big(\|\dib u\|_0^2+\|\dib^* u\|_0^2+\|u\|_0^2+\|u\|_0\|u\|_1 \big). 
\]
\end{proof}

%
%

\section{Proof of Theorem \ref{thm:Khanh-Andy1} }\label{sec:proof of thm1}
\subsection{Kohn's weighted theory}
We first recall the results of Kohn's weighted theory. We let $\H_{p,q}(\Om,t\lambda)$
be the $(p,q)$-forms annihilated by $\dbar$ and $\dbar^{*,t}$ in $L^2_{p,q}(\Om,t\lambda)$.
\begin{theorem}\label{thm:L^2-theory}
Let $\Om$ be a smooth, bounded, pseudoconvex domain in a complex manifold $M$ that admits a smooth function $\lambda$ that is strictly plurisubharmonic 
$(p_0,q_0)$-forms in a neighborhood of $b\Om$. If $0\le p\le n$ and $q_0 \leq q \leq n$ then the following hold:
\begin{enumerate}[(i)]
\item The $L^2$ basic estimate holds on $L^2_{p,q}(\Omega)$: namely, there exists $c>0$ so that
\begin{equation} \label{L2basic new} 
\|u\|_0^2\le c(\|\dib u\|_0^2+\|\dbars u\|_0^2+\|u\|^2_{-1}).
\end{equation}
holds for all $u\in\Dom(\dib)\cap\Dom(\dbars)$.
	
\item The operators $\dbar: L^2_{p,\tilde q}(\Omega)\to L^2_{p,\tilde q+1}(\Omega)$ and
$\dbars: L^2_{p,\tilde q+1}(\Omega)\to L^2_{p,\tilde q}(\Omega)$ have closed range when $\tilde q = q$ or $q-1$.
Additionally, $\Box:L^2_{p,q}(\Omega)\to L^2_{p,q}(\Omega)$ has closed range.
	
\item  The space of harmonic forms $\mathcal H_{p,q}(\Omega)$ is finite dimensional. Additionally, there exists a constant $c>0$ so that  
\begin{equation} \label{L2basic} 
\|u\|_0^2\le c(\|\dbar u\|_0^2+ \|\dbars u\|_0^2)
\end{equation}
holds for all $u\in\Dom(\dbar)\cap\Dom(\dbars)\cap \H^{\perp}_{p,q}(\Omega)$.
	
\item The operators $N_{p,q}$, $\dbars N_{p,q}$, $N_{p,q}\dbars$, $\dbar N_{p,q}$, $N_{p,q}\dbar$, 
$I - \dbars \dbar N_{p,q}$, $I - \dbars N_{p,q} \dbar$, $I - \dbar \dbars N_{p,q}$, 
$I - \dbar N_{p,q} \dbars$ are $L^2$ bounded. 
In the case $q=1$, the operators $N_{(p,0)}:= \dbars N_{p,1}^2 \dbar$  
and hence $N_{(p,0)}$, $\dib N_{(p,0)}$ are continuous on  $L^2_{p,q}(\Om)$.


\end{enumerate}

Furthermore, for $s\geq 0$, there exists $T_s\geq 0$
so that for $t\ge T_s$ the following hold:
\begin{enumerate}[(i)]
\setcounter{enumi}{5}
\item The space of harmonic forms $\H_{p,q}(\Om,t\lambda) \subset H^s_{p,q}(\Omega)$ and 
is finite dimensional. 

\item The weighted $\dbar$-Neumann Laplacian $\Box^t = \dbar\dbar^{*,t} + \dbar^{*,t}\dbar$ has closed range in both 
$L^2_{p,q}(\Omega,e^{-t\lam})$ and $H^s_{p,q}(\Omega)$ if $t\geq T_s$.

\item The $\dbar$-Neumann operator $N_{p,q}^t$ and the canonical solution operators $\dbar^{*,t} N_{p,q}^t$, 
$N_{p,q}^t\dbar^{*,t}$, $\dbar N_{p,q}^t$, $ N_{p,q}^t\dbar$,  $I - \dbar^{*,t}\dbar N_{p,q}^t$, $I - \dbar^{*,t} N_{p,q}^t \dbar$ 
are exactly regular in the $H^s$-spaces.
\end{enumerate}
\end{theorem}

\begin{proof} Here we sketch the main idea of the argument.  First, it suffices to concentrate on forms defined on $\supp\lambda$ because
ellipticity and interior regularity make Theorem \ref{thm:L^2-theory} automatic for forms supported away from $\bdy\Om$.

Next, since $\lambda$ is a strictly plurisubharmonic function acting on $(p_0,q_0)$-forms, it is also strictly plurisubharmonic on
$(p,q)$-forms also holds for all $0\le p\le n$ and $q_0\le q\le n$. 
We then apply the basic estimate with the weight $\phi=t\lambda$ with a sufficient large $t$ to obtain
\begin{equation}\label{eqn:L2 weight}
\no{u}_{L^2_{t\lambda}}^2\le \frac{c}{t}\left(\no{\dib u}_{L^2_{t\lambda}}^2+\no{\dib^{*,t}u}_{L^2_{t\lambda}}^2\right)+c_t\no{u}^2_{-1}
\end{equation} 	 
holds for all $u\in\Dom(\dib)\cap\Dom(\dib^{*,t})\cap L^{2}_{p,q}(\Om)$. 
Once you have the strong closed range estimates \eqref{eqn:L2 weight}, the conclusions of 
the theorem are well known. See, for example, \cite{Str10,HaRa22SCRE,ChHaRa24} and \cite[Section 5.3]{ChSh01} for the $L^2$-theory on complex manifolds.
\end{proof}

\subsection{\emph{A priori} estimate for  Theorem \ref{thm:Khanh-Andy1} }
\begin{lemma}	Let $\Om$ be a smooth, bounded, pseudoconvex domain in a complex manifold $M$ that admits a strictly plurisubharmonic function $\lambda$ on $(p_0,q_0)$-forms. 
If  $u\in L^2_{p,q}(\Om)\,\cap\Dom(\Box)$ with $0\le p\le n$,  $q_0 \leq q \leq n$ then
there exists $t>0$ such that 
\begin{align}
u&= P_{p,q}e^{-t\lambda} N_{p,q}^{t\lambda}\dib\left( e^{t\lambda}\dib^* u\right)+(I-P_{p,q})\dib^{*,t\lambda}N_{p,q+1}^{t\lambda}\dib u+H_{p,q}u 
+  P_{p,q} e^{-t\lambda} H^{t\lambda}_{p,q} e^{t\lambda} (I-P_{p,q}) u
\label{eqn:equality}
\end{align}
\end{lemma}

\br The $q=n$ case is not interesting as the $\dbar$-Neumann problem is the Dirichlet problem. Additionally, $P_{p,n}=I$, the term with $N_{p,n+1}$ is always $0$ (so we do not have to make a special
definition that $N_{p,n+1}=0$), and the result is straight forward. Thus, we can assume that $q \leq n-1$. 
\er
\br \label{rem:can swap P_q for H_q}
We observe that $\dbar P_{p,q} e^{-t\lambda} H^{t\lambda}_{p,q} e^{t\lambda}=0$ and
\[
\dbars P_{p,q} e^{-t\lambda} H^{t\lambda}_{p,q} e^{t\lambda}
=\dbars  e^{-t\lambda} H^{t\lambda}_{p,q} e^{t\lambda} = e^{-t\lambda}  \dbar^{*,t-\lam} H^{t\lambda}_{p,q} e^{t\lambda} =0.
\]
Thus, $ P_{p,q} e^{-t\lambda} H^{t\lambda}_{p,q} e^{t\lambda} =  H_{p,q} e^{-t\lambda} H^{t\lambda}_{p,q} e^{t\lambda}$.
\er

\begin{proof} The proof of \eqref{eqn:equality} adapts ideas from \cite{BoSt90}. 
We start by establishing an identity in the spirit of Boas-Straube \cite{BoSt90} for $\dib N_{p,q}$ and $\dib^* N_{p,q}$ for $\Om\subset M$. 
Let $\varphi \in L^2_{p,q}(\Om)$
and $f=P_{p,q}\varphi$. Then $f$ is a $\dib$-closed $(p,q)$-form 
so there exist $t>0$ such that $v:=\dib^{*,t\lambda} N^{t\lambda}_{p,q} f$ is a solution of the 
$\dib v=(I - H^{t\lambda}_{p,q})f = (I - H^{t\lambda}_{p,q})P_{p,q}\vp$. It follows
\begin{align*}
\dib^* N_{p,q}P_{p,q}\varphi 
&=\dib^* N_{p,q}  f
=\dib^* N_{p,q}\dib  v + \dbars N_{p,q} H^{t\lambda}_{p,q} f \\
&= \dib^* N_{p,q}\dib \dib^{*,t\lambda} N^{t\lambda}_{p,q}\dbar v + \dbars N_{p,q} H^{t\lambda}_{p,q} f\\
&= \dib^* N_{p,q}\dib \dib^{*,t\lambda} N^{t\lambda}_{p,q}(I - H^{t\lambda}_{p,q})P_{p,q}\vp + \dbars N_{p,q} H^{t\lambda}_{p,q} f \\
&= (I-P_{p,q-1}) \dib^{*,t\lambda} N^{t\lambda}_{p,q}P_{p,q}\vp + \dbars N_{p,q} H^{t\lambda}_{p,q} P_{p,q}\varphi
\end{align*}

In other words, the following identity 
\[
\dib^* N_{p,q}\varphi =(I-P_{p,q-1}) \dib^{*,t\lambda} N^{t\lambda}_{p,q}P_{p,q}\varphi  + \dbars N_{p,q} H^{t\lambda}_{p,q} P_{p,q}\varphi
\]
holds since $\dib^* N_{p,q}P_{p,q}\varphi=\dib^* N_{p,q}\varphi$. We then take the $L^2$-adjoint of $\dib^* N_{p,q}$ and $(I-P_{p,q}) \dib^{*,t\lambda} N^{t\lambda}_{p,q}P_{p,q}$
to obtain 
\[
N_{p,q}\dib \psi =P_{p,q}e^{-t\lambda}N^{t\lambda}_{p,q}\dbar\big(e^{t\lambda}(I- P_{p,q-1})\psi\big)
+ P_{p,q} e^{-t\lambda} H^{t\lambda}_{p,q} e^{t\lambda} N_{p,q} \dbar \psi
\]
for any $\psi\in L^2_{p,q-1}(\Om)$. Observe that since $P_{p,q} = I - \dbars \dbar N_{p,q}$, it follows that $\dbars P_{p,q} = \dbars$ so
\[
\dbars P_{p,q} e^{-t\lambda} H^{t\lambda}_{p,q}
= \dbars e^{-t\lambda} H^{t\lambda}_{p,q} =  e^{-t\lambda} \dbar^{*,t\lambda} H^{t\lambda}_{p,q} =0.
\]
Thus, $P_{p,q} e^{-t\lambda} H^{t\lambda}_{p,q}e^{t\lambda}: L^2_{p,q}(\Om) \to \H_{p,q}(\Om)$.

We now ready for the proof of the formula \eqref{eqn:equality}. If $u\in L^2_{p,q}(\Om) \cap\Dom(\Box)$, we have
\begin{eqnarray}
\begin{split}
u&=N_{p,q}\Box u+H_{p,q}u\\
&= \big(N_{p,q}\dib \big)\big(\dib^* u\big)+\big(\dib^* N_{{p,q+1}}\big)\big( \dib u\big)+H_{p,q}u\\
&=P_{p,q}e^{-t\lambda}N^{t\lambda}_{p,q}  \dbar\big(e^{t\lambda}(I- P_{p,q-1})\dbars u\big)
+ P_{p,q} e^{-t\lambda} H^{t\lambda}_{p,q} e^{t\lambda} N_{p,q} \dbar \dbars u \\
&+(I-P_{p,q}) \dib^{*,t\lambda} N^{t\lambda}_{p,q+1}P_{p,q+1}\dbar u  + \dbars N_{p,q+1} H^{t\lambda}_{p,q+1} P_{p,q+1}\dbar u+H_{p,q}u\\
&=P_{p,q}e^{-t\lambda} N_{p,q}^{t\lambda}\dib\left( e^{t\lambda}\dib^* u\right)+(I-P_{p,q})\dib^{*,t\lambda}N_{p,q+1}^{t\lambda}\dib u+H_{p,q}u \\
&+ P_{p,q} e^{-t\lambda} H^{t\lambda}_{p,q} e^{t\lambda} N_{p,q} \dbar \dbars u+ \dbars N_{p,q+1} H^{t\lambda}_{p,q+1} \dbar u
\end{split}
\end{eqnarray}
where the last equality follows by $(I-P_{p,q-1})\dib^* =\dib^* $  and $P_{p,q+1}\dib=\dib$. Since $N_{p,q} \dbar \dbars = I-P_{p,q}$, we have only to investigate the final term
because we do not want $\dbars N_{p,q+1}$ as the final operator in a composition. Fortunately, however, we can exploit  $H^{t\lambda}_{p,q+1} \dbar$ and write
\begin{align*}
\dbars N_{p,q+1} H^{t\lambda}_{p,q+1} \dbar u
&= \dbars N_{p,q+1} \big(I - \dbar^{*,t\lambda} \dbar N^{t\lambda}_{p,q+1} - \dbar\dbar^{*,t\lambda} N^{t\lambda}_{p,q+1}\big) \dbar u \\
&= \dbars N_{p,q+1}\dbar \big(I -  \dbar^{*,t\lambda} N^{t\lambda}_{p,q+1}\dbar\big) u \\
&= \big(I - P_{p,q}\big) P^{t\lambda}_{p,q} u = \big(I - P_{p,q}\big)P_{p,q} P^{t\lambda}_{p,q} u=0.
\end{align*}
The result follows.
\end{proof}

\begin{lemma}\label{lmm:degree extension}	
Assume that the hypothesis of Theorem~\ref{thm:Khanh-Andy1} holds for forms of degree $(p_0,q_0)$. Then it still holds on forms of degrees $(p,q)$ with $0\le p\le n$ and $q_0\le q\le n$. 
\end{lemma}

\begin{proof} It suffices to prove the result for $p=p_0=0$.

We need to understand better the effect of the $\#$-operator on a $(1,0)$-form and forms of a higher degree. 
Given $1 \leq k \leq n$ and a $(0,q+1)$-form $u = \sum_{K\in\I_{q+1}} u_K\, \bom^K$, define the coefficient function $u_{kJ}$ by
$u_{kJ} = \eps^{kJ}_K u_K$
where $\eps^{kJ}_K$ is the sign of the permutation if $\{k\}\cup J = K$ as sets and $\eps^{kJ}_K=0$ otherwise. For such a $(0,q+1)$-form $u$, set
\[
u_k = \sum_{J\in\I_q} u_{kJ}\, \bom^J .
\]
It is known that $u_k \in \Dom(\dbar)\cap\Dom(\dbars)$, e.g., \cite[Proposition 4.5]{Str10}.

Observe that
\[
\bom_k \wedge u_k = \sum_{J\in\I_q} u_{kJ}\, \bom_k \we \bom^J = \sum_{\atopp{K\in\I_{q+1}}{k\in K}} u_K\, \bom^K,
\]
which means
\[
\sum_{k=1}^n \bom_k \we u_k = \sum_{k=1}^n \sum_{\atopp{K\in\I_{q+1}}{k\in K}} u_K\, \bom^K = (q+1) \sum_{K\in\I_{q+1}} u_K\, \bom^K = (q+1)  u.
\]
Next, we compute how $\#$-operator acts across wedge products. Namely, if $\theta$ is a $(1,0)$-form, $1\leq k \leq n$ and $J\in\I_q$, then
\[
\theta \# \big(\bom_k \we \, \bom^J\big)
= \theta \#\Big(\sum_{K\in\I_{q+1}} \eps^{kJ}_K \, \bom^K\Big)
= \sum_{\atopp{K\in\I_{q+1}}{J'\in\I_q}} \eps^{kJ}_K \sum_{j=1}^n \eps^{jJ'}_K \theta_j \bom^{J'}.
\]

We investigate the product $\eps^{kJ}_K\eps^{jJ'}_K = \eps^{kJ}_{jJ'}$. If $k=j$,  
then $\eps^{kJ}_{jJ'}=1$ if and only if $J=J'$. If $j \neq k$, then
\[
\eps^{kJ}_K\eps^{jJ'}_K = \eps^{kJ}_{jJ'} = \sum_{I\in\I_{q-1}} \eps^{kjI}_{jkI}\eps^{kI}_{J'}\eps^J_{jI} = -\sum_{I\in\I_{q-1}}\eps^{kI}_{J'}\eps^J_{jI}
\]
Given these computations, we see
\[
\theta \# \big(\bom_k \we \, \bom^J\big) = \theta_k \, \bom^J - \sum_{\atopp{J'\in\I_q}{I\in\I_{q-1}}}\sum_{\atopp{j=1}{j\neq k}}^n \eps^{kI}_{J'}\eps^J_{jI} \theta_j \, \bom^{J'}
= (\theta\# \bom_k)\we \bom^J - \bom_k \we \Big(\sum_{I\in\I_{q-1}}\sum_{\atopp{j=1}{j\neq k}}^n \eps^J_{jI} \theta_j \, \bom^I\Big).
\]
From this calculation it is immediate that 
\begin{align*}
&\frac{1}{q+1} \theta \# u = \theta \# \Big(\sum_{k=1}^n \bom_k \wedge u_k\Big) = \sum_{k=1}^n(\theta \# \bom_k) u_k - \sum_{k=1}^n \Big( \bom_k \we \big(\theta\# u_k\big)\Big) \\
&= \sum_{k=1}^n\theta_k u_k - \sum_{k=1}^n \Big( \bom_k \we \big(\theta\# u_k\big)\Big)
= \sum_{J\in\I_q} \sum_{k=1}^n \theta_k u_{kJ}\, \bom^J - \sum_{k=1}^n \Big( \bom_k \we \big(\theta\# u_k\big)\Big) \\
&= \theta \# u- \sum_{k=1}^n \Big( \bom_k \we \big(\theta\# u_k\big)\Big)
\end{align*}
Thus, good estimates for $\alpha_\eps \# u_k$ imply good estimates for $\bom_k \we (\alpha_\eps\# u_k)$ which in turn will imply the desired estimates for
$\alpha_\eps\# u$.
\end{proof}

\begin{theorem}\label{thm:apriori thm1} 
Assume that the hypothesis of Theorem~\ref{thm:Khanh-Andy1} holds for forms of degree $(p,q)$ and that the Bergman projection 
$P_{p,q}$ is exactly regular. 
Then $\H_{p,q}(\Om)\subset C^\infty_{p,q}(\bar\Om)$,  
and for any $s\in \mathbb N$,  there exists $\delta_s>0$  such that the 
\emph{a priori} estimate for the operator $\Box^{\delta}:=\Box+\delta T^*T$	
\begin{equation}
\|u\|^2_s+\|\dbar u\|^2_s+\|\dbars u\|^2_s+\delta\|u\|^2_{s+1}+\|\dbar\dbars u\|^2_s+\|\dbars\dbar u\|^2_s+\delta^2\|u\|^2_{s+2}\le c_s\left(\|\Box^{\delta}u\|_s^2+\|u\|_0^2\right),
\end{equation}
holds for any $u\in C^\infty_{p,q}(\bar\Om)\cap \T{Dom}(\Box)$ and any $\delta\in [0,\delta_{s})$.
\end{theorem}

\begin{proof} 	
Let $\mathcal A_\delta^s(u)$ be the expression defined by
\[
\mathcal A_\delta^s(u):= \|u\|^2_s+\|\dbar u\|^2_s+\|\dbars u\|^2_s+\delta\|u\|^2_{s+1}+\|\dbar\dbars u\|^2_s+\|\dbars\dbar u\|^2_s+\delta^2\|u\|^2_{s+2}.
\]
We prove the theorem by induction, and our argument will show that $\H_{p,q}(\Om)\subset C^\infty_{p,q}(\bar\Om)$ and
\begin{equation}\label{eqn:induction}
A_\delta^s(u)\le c_s \|\Box^\delta u\|^2_s+A_\delta^{s-1}(u)
\end{equation}
for any $s\ge 1$, and  
\begin{equation}\label{eqn:induction s=0}
A_\delta^0(u)\le c\big(\|\Box^\delta u\|_0^2+\|u\|_0^2\big).
\end{equation} 
We show \eqref{eqn:induction} by the following argument:

\textbf{Step 1: The proof that $\H_{p,q}(\Om)\subset C^\infty_{p,q}(\bar\Om)$ and an estimation of} $\|u\|_s^2$.
For $s>0$ fixed and $t$ sufficiently large, we combine \eqref{eqn:equality} with Theorem~\ref{thm:L^2-theory} (viii) and the facts that
$P_{p,q}$, $N^{t\lambda}_{p,q}\dib$, and $\dib^{*,t\lambda}N^{t\lambda}_{p,q+1}$ are continuous on $H^s$ to observe that if
$u\in C^\infty_{p,q}(\bar\Om)\cap\Dom(\dbars)$ then $H_{p,q}u\in H^s_{p,q}(\bar\Om)$. By Remark \ref{rem:can swap P_q for H_q}, the regularity of $I-P_{p,q}$, and the fact that $s\in\mathbb N$ is arbitrary,
it follows that
$H_{p,q}u \in C^\infty_{p,q}(\bar\Om)$, and
\begin{equation*}
\no{u}^2_s\le c_s(\no{\dib u}^2_s+\no{\dib^* u}^2_s+\no{H_{p,q}u}_s^2),
\end{equation*}
holds for any $u\in C^\infty_{p,q}(\bar\Om)\cap \T{Dom}(\dib^*)$. 
Moreover,  the   finite dimensionality and smoothness of harmonic
forms force that 
\begin{equation}\label{eqn:HstoH0}
\no{H_{p,q}(u)}_s\le c_s\no{H_{p,q}u}_0\le c_s\no{u}_0.
\end{equation}
Thus, \begin{equation}\label{eqnL:Hs estimate fo d and dbar}
\no{u}^2_s\le c_s(\no{\dib u}^2_s+\no{\dib^* u}^2_s+\no{u}_0^2),
\end{equation}
holds for any $u\in C^\infty_{p,q}(\bar\Om)\cap \T{Dom}(\dib^*)$.\\

{\bf Step 2: Estimate  $\no{\dib\dib^* u}^2_s+\no{\dib^* \dib u}^2_s+\delta^2\no{u}^2_{s+2}$.} 
Using \cite[Lemma 2.2]{Str10} as in \cite[(5.47)]{Str10} and letting $\vartheta$ be the formal adjoint of $\dbar$,  
we  obtain (inductively)
\begin{eqnarray}\label{eqn:step1a}\begin{split}
&\no{\dib^*\dib u}^2_s+\no{\dib\dib^* u}^2_s \\ 
&\le c_{s} \left(\no{\dib\dib^*\dib u}^2_{s-1}
+\no{\vartheta\dib\dib^* u}^2_{s-1}+\no{\dib^*\dib u}^2_{s-1}+\no{\dib\dib^* u}^2_{s-1}+\no{T^s\dib^*\dib u}_0^2
+\no{T^s\dbar\dib^* u}_{0}^2\right)\\
&\le c_{s} \left(\no{\Box^\delta u}^2_{s}+A_\delta^{s-1}(u)+\no{T^{s}\dib^*\dib u}_0^2+\no{T^{s}\dbar\dib^* u}_0^2+\delta^2 \no{u}^2_{s+2}\right).
\end{split}\end{eqnarray}
Here, the last inequality follows by 
\begin{align*}
\no{\dbar\dbars\dbar u}^2_{s-1}+\no{\vartheta\dib\dib^* u}^2_{s-1}&=\no{\dib\Box u}^2_{s-1}
+\no{\vartheta\Box u}^2_{s-1}\\
&\le c\no{\Box u}^2_s\le c\left(\no{\Box^\delta u}^2_s+\delta^2\no{T^*T u}^2_s\right)\le  c\left(\no{\Box^\delta u}^2_s+\delta^2\no{u}^2_{s+2}\right).
\end{align*}
To estimate $\delta^2\no{u}^2_{s+2}$, we use Lemma~\ref{lmm:Hs Ts}(ii) with $T_\epsilon$ replaced by $T$ to estimate
\begin{eqnarray}\label{eqn:step1b}
\begin{split}
\no{u}^2_{s+2}&\le c_s\left(\no{\Box u}^2_{s}+\no{u}^2_{s+1}+ \no{T^{s+2}u}_0^2\right)\\
&\leq c_s\left(\no{\Box^\delta u}_s^2+\delta^2 \no{T^*Tu}^2_s+\no{u}^2_{s+1}+ \no{T^{s+2}u}_0^2\right).
\end{split}
\end{eqnarray}
We observe that $\delta^2 \no{T^*Tu}^2_s\le c\delta^2 \no{u}^2_{s+2}$ and it is absorbed by the LHS for sufficient small $\delta$. Moreover,  
\begin{eqnarray}\label{eqn:step1c}
\begin{split}
\no{T^{s+2}u}_0^2\le c\left(\no{T^sT^*T u}_0^2+\no{u}_{s+1}^2\right),
\end{split}
\end{eqnarray}
since $T=T^*+a$ with $a\in C^\infty(\bar\Om)$.  Thus, from \eqref{eqn:step1a}, \eqref{eqn:step1b}, and \eqref{eqn:step1c}, we have 
\begin{eqnarray}\label{eqn:step1d}
\begin{split}
\no{\dib\dib^* u}^2_s+\no{\dib^* \dib u}^2_s+&\delta^2\no{u}^2_{s+2}\le c_s\Big(\no{\Box^\delta u}^2_{s}+A_\delta^{s-1}(u)+\delta^2  \no{u}_{s+1}^2\\
&+\no{T^{s}\dib^*\dib u}_0^2+\no{T^{s}\dib\dib^* u}_0^2+\delta^2\no{T^sT^*Tu}_0^2\Big)\\
&=c_s\Big(\no{\Box^\delta u}^2_{s}+A_\delta^{s-1}(u)+\delta^2  \no{u}_{s+1}^2+\no{T^s \Box^\delta u}_0^2\\
&-2\Rre\left\{(T^s\dib\dib^* u,T^s\dib^*\dib u)+\delta(T^s\dib\dib^* u,T^sT^*Tu)+\delta (T^s\dib^*\dib u,T^sT^*Tu)\right\}\Big)\\
\end{split}
\end{eqnarray}
Since applying $T^s$ to a form $u \in C^\infty_{p,q}(\bar\Om)\cap\Dom(\Box)$ does not affect the boundary 
$\dib$-Neumann condition of $u$, we can commute and integrate by parts the $\dib$, $\dib^*$ and $T^*$ terms 
in the inner product pieces of $-2\Rre(\dots )$ to obtain 
\begin{eqnarray}
\label{eqn:step1e}
\begin{split}
&-2\T{Re }\left((T^s\dib\dib^* u,T^s\dib^*\dib u)+\delta(T^s\dib\dib^* u,T^sT^*Tu)+\delta (T^s\dib^*\dib u,T^sT^*Tu)\right)\\
&=-2\delta\left(\no{T^{s+1}\dib u}_0^2+\no{T^{s+1}\dib^* u}_0^2\right)+\T{good terms}\le \T{good terms}
\end{split}
\end{eqnarray}
where 
\begin{eqnarray}
\label{eqn:step1f}
\begin{split}
\T{good terms}\le c_s\left(\no{\dib u}_s\no{\dib\dbars u}_s+\delta\no{\dib^*u}_s\no{u}_{s+2}+\delta\no{\dib u}_s\no{u}_{s+2}\right)
\end{split}
\end{eqnarray}
Indeed, 
\begin{eqnarray*}
\begin{split}
(T^s\dbar\dbars u, T^s \dbars\dbar u)&=(T^s\dbar\dbar\dbars u, T^s\dib u)+(T^s\dbar\dbars u, [T^s, \dbars]\dbar u)
+([\dbar,T^s]\dbar\dbars u, T^s \dbar u);\\
(T^s\dib\dib^* u,T^sT^*Tu)&=\no{T^{s+1}\dib^* u}_0^2+(T^s\dib^* u,[T^s,T^*]T\dib^* u)+(T^s\dib^* u,[\dib^*,T^sT^*T]u)+([T^s,\dib]\dib^* u,T^sT^*Tu);\\
(T^s\dib^*\dib u,T^sT^*Tu)&=\no{T^{s+1}\dib u}_0^2+(T^s\dib u,[T^s,T^*]T\dib u)+(T^s\dib u,[\dib,T^sT^*T]u)+([T^s,\dib^*]\dib u,T^sT^*Tu).
\end{split}
\end{eqnarray*}
Using the \eqref{eqn:sc-lc} inequality for the upper bound of ``good terms" we can absorb $\no{\dib^*\dib u}^2_s$ and $\delta^2\no{u}_{s+2}^2$ by the   LHS of \eqref{eqn:step1d}. This gives us
\begin{equation}\label{eqn:step1finnal}
\no{\dib\dib^* u}^2_s+\no{\dib^* \dib u}^2_s+\delta^2\no{u}^2_{s+2}
\le  c_s\Big(\no{\Box^\delta u}^2_{s}+A_\delta^{s-1}(u)+\no{\dib u}_s^2+\no{\dib^*u}_s^2+ \delta\no{u}_{s+1}^2\Big)
\end{equation}
The previous inequality is true with $\delta^2  \no{u}_{s+1}^2$ but we only require $\delta\no{u}_{s+1}^2$.
\begin{remark}\label{rmk:no need hypothesis} We remark that in the the proof of \eqref{eqn:step1finnal}, 
we do not require any specific hypothesis of Theorem~\ref{thm:Khanh-Andy1} beyond the pseudoconvexity of $\Om$.
\end{remark}

{\bf Step 3: Estimate $\mathcal A_\delta^{s}(u)$.} By the estimate of $\no{u}_s^2$ and  $\no{\dib\dib^* u}^2_s+\no{\dib^* \dib u}^2_s+\delta^2\no{u}^2_{s+2}$ in Steps 1 and 2, we have 
\begin{eqnarray}\label{eqn:step20}\begin{split}
\mathcal A_\delta^{s}(u)\le c_s\left(\no{\dib u}_s^2+\no{\dib^*u}_s^2+\delta  \no{u}_{s+1}^2+\no{\Box^\delta u}^2_{s}+A_\delta^{s-1}(u)\right).
\end{split}\end{eqnarray}
To estimate $\no{\dib u}_s^2+\no{\dib^*u}_s^2+\delta  \no{u}_{s+1}^2$, we  first use Lemma~\ref{lmm:Hs Ts}(i) for $u$ replaced by $\dib^* u$, $\dib u$, and Lemma  ~\ref{lmm:Hs Ts}(ii), it follows 
\begin{eqnarray}\label{eqn:step2a}\begin{split}
\no{\dib^*u}^2_s+\no{\dib u}^2_s+\delta\no{u}_{s+1}^2
&\leq c_{\eps, s} \left(\no{\dib\dib^* u}^2_{s-1}+\no{\dib^*\dib u}^2_{s-1}+\no{\dib^*u}^2_{s-1}
+\no{\dib u}^2_{s-1}+\delta\no{\Box u}^2_{s-1}+\delta \no{u}^2_s\right)\\
&+c_s\left(\no{T^s_\eps\dib^*u}_0^2+\no{T^s_\eps\dib u}_0^2+\delta \no{T^{s+1}_\eps u}_0^2\right)\\
&\leq c_{\eps, s}\left(\mathcal A_\delta^{s-1}(u)+\delta \no{\Box u}^2_{s-1}+\delta \no{u}_s^2\right)\\
&+c_s\left(\no{T^s_\eps\dib^*u}_0^2+\no{T^s_\eps\dib u}_0^2+\delta \no{T^{s+1}_\eps u}_0^2\right).
\end{split}\end{eqnarray}
Since $T^s_\eps u\in\Dom(\dib^*)$, we have
\begin{align*}
\no{T^s_\eps\dib^*u}_0^2=&([T^s_\eps,\dib^*]u,T^s_\eps\dib^*u)+(T^s_\eps u,[\dib,T^s_\eps]\dib^*u)+(T^s_\eps u,T^s_\eps\dib \dib^*u)\\
&=-(sh_\eps^{-1}\alpha_\eps\# T_\eps^s u, T^s_\eps\dib^*u)
-(T^s_\eps u,sh_\eps^{-1}\bar \alpha_\eps\we T_\eps^s\dib^*u)+(T^s_\eps u,T^s_\eps\dib \dib^*u)\\
&+(T^s_\eps u, \mathcal X_{\eps,s} \dib^* u)-(\mathcal Y_{\eps,s}u,T_\eps^s \dib^* u) \\
&=(T^s_\eps u,T^s_\eps\dib \dib^*u)-2s(h_\eps^{-1} \alpha_\eps\# T_\eps^s u, T^s_\eps\dib^*u)
+(T^s_\eps u, \mathcal X_{\eps,s} \dib^* u)-(\mathcal Y_{\eps,s}u,T_\eps^s \dib^* u) \\
\end{align*}
Using the \eqref{eqn:sc-lc} inequality, we may absorb the $T^s_\eps \dib^*u$ terms, so that
\begin{equation}\label{eqn:step2b}
\no{T^s_\eps\dib^*u}_0^2\le 
c\left(\Rre(T^s_\eps u,T^s_\eps\dib \dib^*u)+s^2\no{\alpha_\eps\# T_\eps^s u}_0^2+\no{T^s_\eps u}_0\no{\mathcal X_{\eps,s} \dib^* u}_0
+\no{\mathcal Y_{\eps,s}u}_0^2\right)
\end{equation}	
Since both $T^s_\eps u, T^s_\eps \dib u\in \T{Dom}(\dib^*)$, a similar calculation shows that
\begin{eqnarray}\label{eqn:step2c}\begin{split}
\no{T^s_\eps\dib u}_0^2
&=(T^s_\eps u,T^s_\eps\dib^* \dib u)-2s(h_\eps T_\eps^s u, \alpha_\eps\# T^s_\eps\dib u)
-(\mathcal X_{\eps,s}u,T_\eps^s \dib u)+ (T^s_\eps u, \mathcal Y_{\eps,s} \dib u)\\
&\leq c\big(\Rre (T^s_\eps u,T^s_\eps\dbars\dbar u)+s\no{\alpha_\eps\# T_\eps^s \dib u}_0 \no{T_\eps^s u}_0+
\no{\mathcal X_{\eps,s}  u}_0^2+\no{T_\eps^s u}_0\no{\mathcal Y_{\eps,s}\dib u}_0\big). \\
\end{split}\end{eqnarray}
Since $T_\eps=h_\eps T+\bar X_\eps+Y_\eps$, where $X_\eps,Y_\eps\in T^{1,0}(\Om)$ and $|h_\eps|\approx 1$, it follows from a commutator and 
integration by parts argument that
\begin{eqnarray}\label{eqn:step2d}\begin{split}
\no{T_\eps^{s+1}u}_0^2
&\leq3\left(\no{T_\eps^s h_\eps Tu}_0^2+\no{T_\eps^s \bar X_\eps u}_0^2+\no{T_\eps^s Y_\eps u}_0^2\right)\\
&\leq c\no{T^s_\eps T u}_0^2+c_{\eps, s}\left(\no{u}^2_s+\no{\bar X_\eps u}^2_s+\no{Y_\eps u}^2_s\right)\\
&\leq c \Rre (T^s_\eps T^*T u, T^s_\eps u)+c_{\eps,s}\left(\no{\dib u}_s^2+\no{\dib^*u}_s^2+\no{u}_{s+1}\no{u}_s+\no{u}^2_s\right).
\end{split}\end{eqnarray}
From \eqref{eqn:step20}-\eqref{eqn:step2d}, we have therefore established that
\begin{eqnarray}\label{eqn:step2e}\begin{split}
\mathcal A_\delta^{s}(u)
&\leq  c_s\Big(\Rre(T_\eps^s\Box^\delta u,T_\eps^s u)+s^2\no{\alpha_\eps\# T^s_\eps u}_0^{2}
+s\no{\alpha_\eps\# T_\eps^s \dib u}_0\no{T_\eps^s u}_0\\
&+\no{T^s_\eps u}_0\left(\no{\mathcal X_{\eps,s}\dib^*u}_0
+\no{\mathcal Y_{\eps,s}\dib u}_0\right)+\left(\no{\mathcal X_{\eps,s}u}_0^2+\no{\mathcal Y_{\eps,s}u}_0^2\right)\Big)\\
&+c_{\eps, s}\left(\mathcal A_\delta^{s-1}(u)+\no{\Box^\delta u}_s^2+\delta^3\no{u}^2_{s+1}+\delta\no{\dib u}^2_s
+\delta\no{\dib^* u}_s^2+\delta\no{u}_{s+1}\no{u}_s+\delta \no{u}_s^2\right).
\end{split}\end{eqnarray}
By (\ref{eqn:sc-lc}),
$$ \delta c_{\eps,s}\no{u}_{s+1} \no{u}_s\le \frac{\delta}{4}\no{u}^2_{s+1}+\delta c^2_{\delta,\eps}\no{u}^2_{s}$$ 
the term $\frac{\delta}{4}\no{u}^2_{s+1}$ can be absorbed by the   LHS of \eqref{eqn:step2e}. 
Furthermore, there exists $\delta_{\eps,s}$ such that for any $0\le \delta \le \delta_{\eps,s}$, the sum
$$\delta c_{\eps,s}\no{\dib u}^2_s+\delta c_{\eps,s}\no{\dib^* u}_s^2+ \delta^3 c_{\eps,s}\no{u}_{s+1}^2+\delta(c_{\eps,s}+c^2_{\eps,s})\no{u}_s^2$$
can be absorbed by the LHS of \eqref{eqn:step2e}. We again use (\ref{eqn:sc-lc}) on the terms 
involving $\no{T^s_\eps u}_0$ in the first and second lines of \eqref{eqn:step2e} to obtain
\begin{eqnarray}\label{eqn:step2f}\begin{split}
\mathcal A_\delta^{s}(u)&\leq   c_s\left(\no{\alpha_\eps\# T^s_\eps u}_0^2+\kappa^{-1}\no{\alpha_\eps\# T_\eps^s \dib u}_0^2\right)+\kappa\no{T_\eps^s u}_0^2\\
&+c_{\kappa,s}\left(\no{\mathcal X_{\eps,s}\dib^*u}_0^2+\no{\mathcal Y_{\eps,s}\dib u}_0^2+\no{\mathcal X_{\eps,s}u}_0^2
+\no{\mathcal Y_{\eps,s}u}_0^2\right) +c_{\eps,s,\kappa}\left(\no{\Box^\delta u}^2_{s}+\mathcal A^{s-1}_\delta (u)\right)\\
&\leq   c_s\left(\no{\alpha_\eps\# T^s u}_0^2+\kappa^{-1}\no{\alpha_\eps\# T^s \dib u}_0^2\right)\\
&+\kappa\no{b_\eps^sT^s u}_0^2
+c_{\eps,s,\kappa}\Big(\no{\mathcal X_{\eps,s}\dib^*u}_0^2+\no{\mathcal Y_{\eps,s}\dib u}_0^2
+\no{\mathcal Z_{\eps,s}\dib u}_0^2+\no{\mathcal X_{\eps,s}u}_0^2\\
&+\no{\mathcal Y_{\eps,s}u}_0^2+\no{\mathcal Z_{\eps,s}u}_0^2+\no{\Box^\delta u}^2_{s}+\mathcal A^{s-1}_\delta (u)\Big).
\end{split}\end{eqnarray}
The last two lines are bounded by 
\begin{eqnarray}\label{eqn:step2g}\begin{split}
&c_s\kappa\no{u}^2_s
+c_{\eps,s,\kappa}\Big(\no{\dib u}_s\no{\dib u}_{s-1}
+\no{\dib^*u}_s\no{\dib^*u}_{s-1}+\no{u}_s\no{u}_{s-1}+\no{\Box^\delta u}^2_{s}+\mathcal A^{s-1}_\delta (u)\Big)\\
&\le \kappa c_s\left(\no{\dib u}^2_s+\no{\dib^* u}^2_s+\no{u}^2_s\right)
+c_{\eps,s,\kappa}\Big(\no{\Box^\delta u}^2_{s}+\mathcal A^{s-1}_\delta (u)\Big).
\end{split}\end{eqnarray}
Thus, for sufficiently small $\kappa$,  
$ \kappa c_s\left(\no{\dib u}^2_s+\no{\dib^* u}^2_s+\no{u}^2_s\right)$  will be absorbed by $A_\delta^s(u)$. 
Putting our estimates together, we obtain 
 
\begin{equation}\label{eqn:step2h}
\mathcal A_\delta^{s}(u)
\leq  c_s\left(\no{\alpha_\eps\# T^s  u}_0^2+\no{\alpha_\eps\# T^s \dib u}_0^2\right)
+c_{\eps,s}\Big(\no{\Box^\delta u}^2_{s}+\mathcal A^{s-1}_\delta (u)\Big).
\end{equation}
We now use $\alpha\#$-compactness estimate for $T^s u$ and $T^s\dib u$ (which is justified
since $u\in\Dom(\Box)$ and $T$ is tangential so both $T^su,T^s\dbar u\in\Dom(\dbars)$), 
\begin{eqnarray}\label{eqn:step2i}\begin{split}
&\no{\alpha_\eps\# T^s  u}_0^2+\no{\alpha_\eps\# T^s \dib u}_0^2 \\
&\leq \epsilon \left(\no{\dib T^s u}_0^2+\no{\dib^* T^s u}_0^2+\no{\dib T^s\dib u}_0^2+\no{\dib^* T^s\dib u}_0^2\right)
+c_\eps\left(\no{T^s u}^2_{-1}+\no{T^s \dib u}^2_{-1}\right)\\
&\leq \epsilon \Big(\no{T^s  \dib u}_0^2+\no{ T^s\dib^* u}_0^2+\no{ T^s\dib^* \dib u}_0^2
+\no{[\dib, T^s] u}_0^2+\no{[\dib^*,T^s] u}_0^2+\no{[\dib, T^s]\dib u}_0^2+\no{[\dib^*, T^s]\dib u}_0^2\Big)\\
&+c_{\eps,s}\left(\no{u}^2_{s-1}+\no{\dib u}^2_{s-1}\right)\\
&\leq \eps c_s\Big(\no{\dib u}_s^2+\no{\dib^* u}_s^2+\no{\dib^*\dib u}^2_s+\no{u}_s^2\Big)+c_{\eps,s}\mathcal A_\delta^{s-1}(u).
\end{split}\end{eqnarray}

For sufficiently small  $\eps$, the term $\eps c_s\Big(\no{\dib u}_s^2+\no{\dib^* u}_s^2+\no{\dib^*\dib u}^2_s+\no{u}_s^2\Big)$ can be absorbed by $A_\delta^{s}(u)$. This completes the proof of \eqref{eqn:induction}. 
The estimate for $A_\delta^0(u)\le c\left( \no{\Box^\delta u}^2_0+\no{u}_0\right)$ follows easily by an integration by parts following an
application of (\ref{eqn: T esilon 1}) and (\ref{eqn: T epsilon s}).
\end{proof}

\begin{remark} A very similar computation to the estimate of $\|\dbars u\|_s^2 + \|\dbar u\|_s^2+\delta\|u\|_{s+1}^2$ 
starting in \eqref{eqn:step2a} but with $\delta=0$ produces the inequality
\begin{equation}\label{eqn:good u,dbar u, dbars u with inner product}
\|u\|_s^2 + \|\dbars u\|_s^2 + \|\dbar u\|_s^2 
\leq c_s \Rre\big(T^s_\eps \Box u,T^s_\eps u\big) + c_{\eps,s}\big(\|\Box u\|_{s-1}^2 + \|u\|_0^2\big)
\end{equation}
Indeed, from \eqref{eqn:step2a}, we have
\begin{align*}
\|\dbars u\|_s^2 + \|\dbar u\|_s^2
\leq c_s\big(\no{T^s_\eps\dib^*u}_0^2+\no{T^s_\eps\dib u}_0^2\big) + c_{\eps, s}\mathcal A_0^{s-1}(u).
\end{align*}
Adding together \eqref{eqn:step2b} and \eqref{eqn:step2c} and bounding the error terms by \eqref{eqn:step2i} and 
\eqref{eqn:X,Y,Z ests} and the small constant/large constant inequality (\ref{eqn:sc-lc}) and Lemma \ref{lmm:Hs Ts}
and \eqref{eqn:induction}, we obtain
\begin{align*}
\no{T^s_\eps\dib^*u}_0^2+\no{T^s_\eps\dib u}_0^2
&\leq C_s\Big[ \Rre \big(T^s_\eps \Box u,T^s_\eps u\big) + s^2\|\alpha_\eps \# T^s_\eps u\|_0^2
+ \|T^s_\eps u\|_0\big(\|\opX_{\eps,s}\dbars u\|_0 + \|\opY_{\eps,s}\dbar u\|_0\big) \\
&+ \|\opX_{\eps,s}u\|_0^2 + \|\opY_{\eps,s}u\|_0^2 + s\|\alpha_\eps\#T^s_\eps\dbar u\|_0\|T^s_\eps u\|_0\Big]
\end{align*}
and proceed as above.
\end{remark}

We are now ready to prove Theorem~\ref{thm:Khanh-Andy1}. 
\begin{proof}[Proof of Theorem~\ref{thm:Khanh-Andy1}] We prove this theorem by a
downward induction on $q$, for $0\le q\le n$. 
By Lemma~\ref{lmm:degree extension}, the hypothesis of Theorem~\ref{thm:Khanh-Andy1} holds for forms of degree $(p,q)$ with 
$0\le p\le n$ and $q_0\le q\le n$. In the top degree, $P_{p,n}=I$ is exactly regular. Therefore, we assume that $P_{p,q'}$ 
is exactly regular for $q_0 \leq q \leq q' \leq n-1$.
By Theorem~\ref{thm:apriori thm1} and Theorem~\ref{thm:regularization},  for any $q \leq q' \leq n$,
$\H_{p,q'}(\Om)\subset C^\infty_{p,q'}(\bar\Om)$ 
and  $N_{p,q'}$, $\dib N_{p,q'}$, $\dib^* N_{p,q'}$, $\dib^*\dib N_{p,q'}$, and  $\dib\dib^* N_{p,q'}$ are exactly regular. 
That $N_{p,q}\dib^*$, $\dib N_{p,q}\dib^*$ are exactly regular follows from the equalities 
$N_{p,q}\dib^* =\dib^* N_{p,q+1}$, $\dib N_{p,q}\dib^* =\dib \dib^* N_{p,q+1}$, 
and $\dib^* N_{p,q+1}$,  $\dib \dib^* N_{p,q+1}$ are exactly regular by induction result. 
 
Finally, we prove that $N_{p,q}\dib$ and $\dib^* N_{p,q}\dib$ are exactly regular and 
hence the next step of the induction assumption holds, namely, the Bergman projection $P_{p,q-1}=I-\dib^* N_{p,q}\dib$ is exactly regular. 
Let $\vp\in C^\infty_{p,q-1}(\bar\Om)$,  the regularity of $N_{p,q}$ implies $N_{p,q}\dib \vp  \in C^\infty_{p,q}(\bar\Om)\cap \T{Dom}(\dib^*)$.  
Using \eqref{eqnL:Hs estimate fo d and dbar} for $u=N_{p,q}\dib \vp$, we see that
\begin{eqnarray}\label{eqn:new thm1}\begin{split}
\no{N_{p,q}\dib \vp}^2_s
&\leq c_s\left(\no{\dib N_{p,q}\dib \vp}^2_s+\no{\dib^* N_{p,q}\dib \vp}^2_s+\no{N_{p,q}\dib \vp}_{0}\right)\\
&\leq c_s\left(\no{\dib^* N_{p,q}\dib \vp}^2_s+\no{ \vp}_0^2\right),
\end{split}\end{eqnarray}
where the last inequality follow by $\dib N_{p,q}\dib \vp=N_{p,q+1}\dib \dib \vp=0$ and the $L^2$-basic estimate.

By using the estimate \eqref{eqn:good u,dbar u, dbars u with inner product} with $u=N_{p,q}\dbar\varphi\in C^\infty_{0,q}(M)$ we obtain
\begin{align*}
\no{N_{p,q}\dbar\varphi}_s^2+\no{\dbar^*N_{p,q}\dbar\varphi}_s^2 
&\le c \Rre(T^s_\eps \Box  N_{p,q}\dbar \varphi, T^s_\eps N_{p,q}\dbar\varphi)+c_{\eps,s}\left(\no{\Box N_{p,q}\dbar\varphi}^2_{s-1}
+\no{N_{p,q}\dbar\varphi}^2_{L^2}\right)\\
&\le c\Rre(T^s_\eps(I-H_{p,q})  \dbar \varphi, T^s_\eps N_{p,q}\dbar\varphi)+c_{\eps,s}\left(\no{(I-H_{p,q}) \dbar\varphi}^2_{s-1}+\no{\varphi}^2_{L^2}\right)\\
&\le c\Rre\Big((T^s_\eps \varphi, T^s_\eps \dbar^*N_{p,q}\dbar\varphi)+([T^s_\eps ,\dbar]\varphi,  T^s_\eps N_{p,q}\dbar\varphi)+(T^s_\eps\varphi,[\dbar^*,T^s_\eps] N_{p,q}\dbar\varphi)\\
&-(T^*_\eps T^{s}_\eps H_{p,q}\dbar\varphi, T^{s-1}_\eps N_{p,q}\dbar\varphi) \Big)+c_{\eps,s}\left(\no{\dbar\varphi}^2_{s-1}+\no{\varphi}^2_{L^2}\right)\\
&\le\T{sc}\left(\no{N_{p,q}\dbar\varphi}_s^2+\no{\dbar^*N_{p,q}\dbar\varphi}_s^2\right)
+c_\eps\T{lc}\left(\no{\varphi}_s^2+\no{H_{p,q}\varphi}_{s+1}^2+ \no{H_{p,q}\dbar\varphi}_{s+1}^2\right).
\end{align*} 
We control the $H_{p,q}$ terms by \eqref{eqn:HstoH0}, and the proof is complete.
\end{proof}

\section{Proof of Theorem \ref{thm:Khanh-Andy2} and Theorem \ref{thm:Khanh-Andy3}}
\label{sec:Proofs}
In this section, we study the regularity of the $\dib$-Neumman problem in domains $\Om\subset M$ when the ambient manifold $M$ is not necessarily Stein.  
That means there is no strictly plurisubharmonic function  acting on $(p,q)$-forms in a neighborhood of the boundary $b\Om$.  However, we assume that the $L^2$ basic estimate holds or 
$M$ is a K\"ahler manifold with positive  holomorphic bisectional curvature.

\subsection{Proof of Theorem \ref{thm:Khanh-Andy2}}
\begin{theorem}\label{thm:apriori thm2}  Assume that the hypothesis of Theorem~\ref{thm:Khanh-Andy2} holds for forms of degree $(p,q)$.  
Then, for any $s\in \mathbb N$,  there exists $\delta_s>0$  such that the \emph{a priori} estimate for the elliptic operator 
$\Box^{\delta}:=\Box+\delta T^*T$,	
	\begin{equation}
	\no{u}^2_s+\no{\dib u}^2_s+\no{\dib^* u}^2_s+\no{\dib\dib^* u}^2_s+\no{\dib^*\dib u}^2_s\le c_s\left(\no{\Box^{\delta}u}_s^2+\no{u}_0^2\right),
	\end{equation}
	holds for any $u\in C^\infty_{p,q}(\bar\Om)\cap \T{Dom}(\Box)$ and any $\delta\in [0,\delta_{s})$.
\end{theorem}

\begin{proof} 
As in Theorem~\ref{thm:apriori thm1}, we let
\[
\mathcal A_\delta^s(u):=\no{u}^2_s+\no{\dib u}^2_s+\no{\dib^* u}^2_s+\delta\no{u}^2_{s+1}
+\no{\dib\dib^* u}^2_s+\no{\dib^* \dib u}^2_s+\delta^2\no{u}^2_{s+2}.
\]
and  prove that
\begin{equation}
A_\delta^s(u)\le c_s \no{\Box^\delta u}^2_s+A_\delta^{s-1}(u)
\end{equation}
for any $s\ge 1$ since it is easy to see that  $A_\delta^0(u)\le c\no{\Box^\delta u}_0^2+\no{u}_0^2$. 
We start by recalling Remark \ref{rmk:no need hypothesis} (from Step 2 of the proof of Theorem~\ref{thm:apriori thm1}) 
that we only need pseudoconvexity to bound $A_\delta^s(u)$ from above by 
\begin{equation}\label{eqn:As thm2 Hs}
A_\delta^s(u)\le c_s\left(\no{u}^2_s+\no{\dib u}^2_s+\no{\dib^* u}^2_s+\delta \no{u}^2_{s+1}+\no{\Box^\delta u}^2_s+A_\delta^{s-1}(u)\right).
\end{equation}
We use Lemma~\ref{lmm:Hs Ts} on the first four terms in RHS of \ref{eqn:As thm2 Hs} to transfer the full $\no{\cdot}_s$-norm 
to partial derivative $T^s_\eps$,
\begin{equation}\label{eqn:As thm2 Ts}
A_\delta^s(u)\le c_s\left(\no{T_\eps^su}_0^2+\no{T_\eps^s \dib u}_0^2+\no{T_\eps^s\dib^* u}_0^2
+\delta \no{T_\eps^{s+1} u}_0^2\right)+c_{\eps,s}\left(\no{\Box^\delta u}^2_s+A_\delta^{s-1}(u)\right).
\end{equation}
We now use the hypothesis
\[
\no{u}_0^2+\frac{1}{\eps}\left(\no{\alpha_\eps \# u}_0^2+\no{\bar \alpha_\eps\wedge  u}_0^2\right)\le c\left(\no{\dib u}_0^2
+\no{\dib^* u}_0^2\right)+c_\eps \no{u}^2_{-1}
\]
with $u$ replaced by $T^s_\eps u\in \Dom(\dib^*)$ to obtain
\begin{eqnarray}\label{eqn: estimate Ts eps}\begin{split}
\no{T^s_\eps u}_0^2&+\frac{1}{\eps}\left(\no{\alpha_\eps \# T^s_\eps  u}_0^2
+\no{\bar \alpha_\eps\wedge  T^s_\eps  u}_0^2\right)\le c\left(\no{\dib T^s_\eps  u}_0^2+\no{\dib^*T^s_\eps  u}_0^2\right)+c_\eps \no{T^s_\eps u}^2_{-1}\\
&\leq c\left(\no{T^s_\eps \dib   u}_0^2+\no{T^s_\eps \dib^* u}_0^2+\no{[\dib, T^s_\eps]  u}_0^2
+\no{[\dib^*,T^s_\eps]  u}_0^2\right)+c_{\eps,s} \no{ u}^2_{s-1}\\
&\leq c_s\left(\no{T^s_\eps \dib   u}_0^2+\no{T^s_\eps \dib^* u}_0^2+\no{\alpha_\eps \# T^s_\eps  u}_0^2
+\no{\bar \alpha_\eps\wedge  T^s_\eps  u}_0^2\right)+c_{\eps,s} A^{s-1}_\delta(u),
\end{split}\end{eqnarray}
where the last inequality follows by Lemma~\ref{lmm:commutators with Ts}. 
For sufficiently small $\eps$, the LHS can absorb the term $\no{\alpha_\eps \# T^s_\eps  u}_0^2+\no{\bar \alpha_\eps\wedge  T^s_\eps  u}_0^2$ 
in the last line of \eqref{eqn: estimate Ts eps}. Combining this estimate with \eqref{eqn:As thm2 Ts} produces
\begin{eqnarray}\label{eqn:As++}\begin{split}
A_\delta^s(u)+&\frac{1}{\eps}\left(\no{\alpha_\eps \# T^s_\eps  u}_0^2+\no{\bar \alpha_\eps\wedge  T^s_\eps  u}_0^2\right)\\
&\leq c_s\left(\no{T_\eps^s \dib u}_0^2+\no{T_\eps^s\dib^* u}_0^2+\delta \no{T_\eps^{s+1} u}_0^2\right)
+c_{\eps,s}\left(\no{\Box^\delta u}^2_s+A_\delta^{s-1}(u)\right).
\end{split}\end{eqnarray}
Similarly to \eqref{eqn:step2b}, \eqref{eqn:step2c}, and \eqref{eqn:step2d}, we have
\begin{eqnarray}\label{eqn:Ts dib* dib u}\begin{split}
\no{T^s_\eps\dib^*u}_0^2
&\leq c\left(\Rre(T^s_\eps u,T^s_\eps\dib \dib^*u)+s^2\no{\alpha_\eps\# T_\eps^s u}_0^2
+\no{T^s_\eps u}_0\no{\mathcal X_{\eps,s} \dib^* u}_0+\no{\mathcal Y_{\eps,s}u}_0^2\right)\\
\no{T^s_\eps\dib u}_0^2
&\leq c\left(\Rre(T^s_\eps u,T^s_\eps\dib^* \dib u)+s^2\no{\alpha_\eps\wedge T_\eps^s u}_0^2
+\no{T^s_\eps u}_0\no{\mathcal Y_{\eps,s} \dib u}_0+\no{\mathcal X_{\eps,s}u}_0^2\right)\\
\no{T_\eps^{s+1}u}_0^2
&\leq c(T^s_\eps T^*T u, T^s_\eps u)+c_{\eps,s}\left(\no{\dib u}_s^2+\no{\dib^*u}_s^2+\no{u}_{s+1}\no{u}_s+\no{u}^2_s\right)
\end{split}\end{eqnarray}
Combine \eqref{eqn:As++} and \eqref{eqn:Ts dib* dib u}, absorb $s^2\no{\alpha_\eps\# T_\eps^s u}_0^2$ and $s^2\no{\alpha_\eps\wedge T_\eps^s u}_0^2$ after choosing small $\epsilon$,  and it follows that
\begin{eqnarray}\label{eqn:final}\begin{split}
A_\delta^s(u)\le c_s\left(\no{\Box^\delta u}_0\no{u}_0^2+A^{s-1}_\delta(u)+\delta A^s_\delta(u)\right).
\end{split}\end{eqnarray}
The desired estimate follows by using the \eqref{eqn:sc-lc} inequality  for the first term and choosing $\delta<\delta_s$ for sufficiently small $\delta_s$. 
\end{proof}
Now we are ready to prove Theorem~\ref{thm:Khanh-Andy2}. 
\begin{proof}[Proof of  Theorem~\ref{thm:Khanh-Andy2}]
	The proof of Theorem~\ref{thm:Khanh-Andy2} follows immediately by of  Theorem~\ref{thm:apriori thm2} and Theorem~\ref{thm:regularization}.	\end{proof}

%
%
\subsection{Proof of  Theorem~\ref{thm:Khanh-Andy3}}
We recall that  a Hermitian form $\omega$ defined in \eqref{eqn:metric defn} is K\"ahler if it  is closed, i.e., $d\omega = 0$. A complex manifold $M$ is K\"ahler if it admits a K\"ahler form. Let $\Theta$ be the holomorphic bisectional curvature $(1,1)$-form  with respect to the K\"ahler metric $\omega$. Denote 
\[
|\bar\nabla u|^2 = \sum_{j=1}^n |\nabla_{\bar L_j} u|^2
\]
where $\{L_1, \cdots, L_n \}$ is an orthonormal  frame for $T^{1,0}M$. The next theorem is the $L^2$ ``basic identity", see, e.g. \cite[Theorem 3.1]{Sha23}.
\begin{theorem}[Bochner-Kodaira-Morrey-H\"ormander] \label{BKMH} Let $\Omega$ be a compact domain in  K\"ahler
manifold $M$ with $C^2$-smooth boundary $b\Om$. For any $(p,q)$-from $u\in \Dom(\dib)\cap \Dom(\dib^*)$, we have 
\[
\no{\dib u}^2_0+ \no{\dib^* u}^2_0 = \no{\bar\nabla u}^2_0 +(\Theta\# u,u) + \int_{b\Om} \la i\di\dib \rho \# u,u\ra dS 
\] 
where $\rho(z)$ is a signed distance function from $z$ to $b\Om$. 
\end{theorem}
It is well-known that in the complex projective space $\C\mathbb P^n$ with the Fubini-Study metric 
$$\omega = i \di\dib \log(1+|z|^2),$$ one has 
$$ \la \Theta\#u, u\ra= q(2n+1)|u|^2$$
for any $(0, q)$-form $u$ on $\C\mathbb P^n$ with $q\ge 1$ \cite[(3.7)]{Sha23}.

We also need the following which is true without the K\"ahler hypothesis on $M$. 
\begin{proposition}\label{prop1}
	Let $M$ be complex manifold and $\Om\subset\subset M$ be a smooth bounded pseudoconvex domain which admits a plurisubharmonic defining function for $\Om$. Then there is a constant $C$ such that for all $\epsilon > 0$  there exists a purely imaginary vector field $T_\epsilon$  with 
	\begin{equation}\label{eqn: alpha}
	C^{-1} \le |\gamma(T_\epsilon)| \le C\quad \T{and}\quad |\alpha_\epsilon|\le \epsilon \quad \text{on}\quad \bdy\Om.
	\end{equation}
\end{proposition}
\begin{proof} The proof can follows exactly as in \cite[the main theorem]{StSu02} or \cite[Section 5.9]{Str10} with the space $\C^n$ replaced by the general complex manifold $M$.  Or it can follow directly by the 
authors' work in \cite[Theorem 1.5]{KhRa20} since $b\Om$ is a plurisubharmonic-oriented CR manifold of hypersurface type. 
\end{proof}

\begin{proof}[Proof of Theorem~\ref{thm:Khanh-Andy3}] Since $M$ is a K\"ahler manifold with positive  holomorphic bisectional curvature acting $(p,q)$-forms, Theorem~\ref{BKMH} implies the $L^2$ basic estimate 
\begin{equation}\label{L2 basic}
\no{u}^2_0 \le c(\no{\dib u}_0^2+\no{\dib^*u}_0^2),
\end{equation}
for any $(p,q)$-form $u\in \Dom(\dib)\cap \Dom(\dib^*)$. It follows that the harmonic space $\mathcal H_{p,q}(\Om)=\{0\}$. 

From the conclusion \eqref{eqn: alpha} in Propositon~\ref{prop1} and the continuity of $\alpha_\epsilon$, one has that for any $\epsilon$ there exists $\delta_\epsilon$ such that 
$ |\alpha_\epsilon|\le 2\epsilon $ on the strip $S_{\epsilon}=\{z\in \Om: -\delta_\epsilon\le  \rho(z)<0\}$. Hence, 
$$\frac{1}{\epsilon}\no{|\alpha_\epsilon| u}^2_{L^2(\Om)}= \frac{1}{\epsilon}\no{|\alpha_\epsilon| u}^2_{L^2(S_\epsilon)}+\frac{1}{\epsilon}\no{|\alpha_\epsilon| u}^2_{L^2(\Omega\setminus S_\epsilon)}\le 4\no{u}^2_{L^2(\Om)}+\tilde C_\epsilon\no{ u}^2_{L^2(\Omega\setminus S_\epsilon)}.$$
for some $\tilde C_\epsilon$. Using \eqref{L2 basic} and the elliptic estimate for the $\dib$-Neumann problem
\[
\tilde C_\epsilon\no{ u}^2_{L^2(\Omega\setminus S_\epsilon)}\le \no{\dib u}^2_{0}+\no{\dib^* u}^2_{0}+ C_\epsilon \no{u}^2_{-1}
\]
for some $C_\epsilon$
we obtain 
\begin{equation}\label{eqn:new1}
\no{u}_0^2+ \frac{1}{\epsilon}\no{|\alpha_\epsilon| u}^2_0 \le c\left(\no{\dib u}^2+\no{\dib^* u}^2\right)+ C_\epsilon \no{u}^2_{-1}.
\end{equation}

We also notice that the inequality  \eqref{eqn:new1} can be rewritten as the main hypothesis \eqref{eqn:main2} in Theorem~\ref{thm:Khanh-Andy2} since
$$\no{|\alpha_\epsilon| u}^2_0 = \no{\bar \alpha_\epsilon\wedge u}^2_0+\no{\alpha_\epsilon\# u}^2_0.$$
Therefore, the proof of Theorem~\ref{thm:Khanh-Andy3} is complete by using Theorem~\ref{thm:Khanh-Andy2}.
\end{proof}

\section{Elliptic Regularization}
In this section, we introduce a new version of elliptic regularization for the $\dib$-Neumann problem in a general setting. Unlike the classical version, we replace the `regular' $\delta \no{\nabla u}_0^2$ with $\delta\no{Tu}_0^2$ to simplify calculations. We believe that our version of elliptic regularization is of independent interest as well.

A proof of the following well known lemma appears in \cite[(17)]{HaRa11}, \cite[p.1077]{StZe15}, and \cite[Lemma 4]{BiSt17}.
\begin{lemma}\label{lem:lem below} Fix $0\le p\le n$ and $1\le q\le n-1$ and $t\in \R$.
	Let $\Om\subset \subset M$ be a bounded domain in a complex manifold $M$. Then the following are equivalent. 
	\begin{enumerate}
		\item[(i)]	The space of harmonic forms $\mathcal H^{t\lambda}_{p,q}(\Om)$ is finite dimensional and the $L^2$ basic estimate on the orthogonal space to $\H^{t\lambda}_{p,q}(\Om)$
		\begin{equation}
		\label{L2basic  lemma} \no{u}_{L_{t\lambda}^2}^2\le c(\no{\dib u}_{L_{t\lambda}^2}^2+\no{\dib^{*,t\lambda}u}_{L_{t\lambda}^2}^2)
		\end{equation}
		holds for all $u\in\Dom(\dib )\cap\Dom(\dib^{*,t\lambda})\cap \left(\H^{t\lambda}_{p,q}(\Om)\right)^{\perp}$.
		\item[(ii)] The $L^2$ basic estimate on $L^2_{t\lambda,(p,q)}(\Om)$
		\begin{equation}
		\label{L2basic new lemma} \no{u}_{L_{t\lambda}^2}^2\le c(\no{\dib u}_{L_{t\lambda}^2}^2+\no{\dib^{*,t\lambda}u}_{L_{t\lambda}^2}^2)+c_t\no{u}^2_{-1}.
		\end{equation}
		holds for all $u\in\Dom(\dib )\cap\Dom(\dib ^*)\cap L^2_{t\lambda,(p,q)}(\Om) $.
	\end{enumerate}
\end{lemma}

The argument to prove regularity from an \emph{a priori} estimate appears in \cite{KhRa20}, but for $\dbarb$. We prove it here for completeness,
as the argument is original.
\begin{theorem}\label{thm:regularization} Let $\Om$ be a bounded domain in a $n$-dimensional complex manifold and let $0\le p\le n$, $1\le q\le n$. Assume that the following two estimates hold:
\begin{enumerate}[i.]
\item The $L^2$-basic estimate:
	\begin{equation}\label{eqn:L2 appendix}
	\no{u}_0^2\le c\left(\no{\dib u}_0^2+\no{\dib^* u}_0^2+\no{u}^2_{-1}\right)
	\end{equation}
	for any $u\in \T{Dom}(\dib)\cap \T{Dom}(\dib^*)\cap L^2_{p,q}(\Om)$, and 
\item The \emph{a priori} estimate  for $\Box^\delta=\Box+\delta T^*T$ for $\delta  \geq 0$: For any $s>0$, there exists a constant $c_s$ so that
		\begin{equation}\label{eqn:ariori appendix}
	\no{u}_s^2\le c_s\left(\no{\Box^\delta u}_s^2+\no{u}_{0}^2\right)
		\end{equation}
	 for  $u\in C^\infty_{p,q}(\bar\Om)\cap \T{Dom}(\Box)$. 
\end{enumerate}

Then $\H_{p,q}(\Om)\subset C^\infty_{p,q}(\bar\Om)$ and the $\dib$-Neumann operator $N_{p,q}$ is exactly regular.
	 
Moreover, if 		
\begin{equation}\label{eqn:ariori appendix 2}
\no{\dib u}_s^2+\no{\dib^* u}_s^2+\no{\dib^*\dib u}_s^2+\no{\dib\dib^* u}_s^2\le c_s\left(\no{\Box u}_s^2+\no{u}_0^2\right)
\end{equation}
holds for any $u\in \T{Dom}(\dib)\cap \T{Dom}(\dib^*)\cap L^2_{p,q}(\Om)$ 
then $\dib N_{p,q}$, $\dib^* N_{p,q}$, $\dib^*\dib N_{p,q}$, and  $\dib\dib^* N_{p,q}$ are exactly regular.  
\end{theorem}

\begin{proof}
Denote by $\Box^{\delta,\nu}=\Box^\delta+\nu I$. From \eqref{eqn:ariori appendix}, there is an $\nu_s$ such that for any $\nu<\nu_s$ the following estimate 
  	\begin{equation}\label{eqn:Box delta nu}
\no{u}_s^2\le c_s\left(\no{\Box^{\delta,\nu} u}_s^2+\no{u}_0^2\right)
\end{equation}
holds for any $u\in C^\infty_{p,q}(\bar\Om)\cap \T{Dom}(\Box)$. Thus, if we define
the quadratic form $Q^{\delta,\nu}$ on $H^1_{p,q}(\Om)  \cap \Dom(\dbars)$ by 
 $$Q^{\delta,\nu}(\cdot,\cdot)=Q(\cdot,\cdot)+\delta (T\cdot,T\cdot)+\nu(\cdot,\cdot)\quad \T{for } \delta,\nu \geq 0,$$
  then, for $\nu>0$
	\begin{equation}\label{eqn:L2 Q delta nu}
	\no{u}_0^2\le \frac{1}{\nu} Q^{\delta,\nu}(u,u)\quad \T{holds for all $u\in H^1_{0,q}(\Om)$.}
	\end{equation} 
	Consequently,
	$\Box^{\delta,\nu}$ is self-adjoint, invertible, and has a trivial kernel. Also, for each $\nu> 0$, the inverse $N_{p,q}^{\delta,\nu}$ satisfies 
	\begin{equation}\label{eqn:L2 G_q delta nu}
	\no{N_{p,q}^{\delta,\nu}\varphi}_{0}\le \frac{1}{\nu}\no{\varphi}_{0}, \quad \T{for all $\varphi\in L^2_{p,q}(\Om)$}.
	\end{equation} 
	When $\delta>0$, we also know 
	\begin{equation}\label{eqn:1 estimates for delta,nu op}
	\no{u}_{1}^2\le c_{\delta,\nu} Q^{\delta,\nu}(u,u)\quad \T{holds for all $u\in H^1_{p,q}(\Om)$.}
	\end{equation} 
    	\textbf{Step 1:}  If $\vp \in  C^\infty_{p,q}(\bar\Om)$ then $N_{p,q}^{0,\nu} \vp \in  C^\infty_{p,q}(\bar\Om)\cap \T{Dom}(\Box)$.
	By elliptic theory, \eqref{eqn:1 estimates for delta,nu op} implies that if $\varphi\in C^\infty_{p,q}(\bar\Om)$, 
	then $N_{p,q}^{\delta,\nu} \varphi\in C^\infty_{p,q}(\bar\Om)\cap \T{\Dom}(\Box^{\delta,\nu})$. We note that 
	$H^1_{p,q}(\Om)\cap \T{\Dom}(\Box)  = \T{\Dom}(\Box^{\delta,\nu})$. 
	We can therefore use (\ref{eqn:Box delta nu}) with $u=N_{p,q}^{\delta,\nu}\varphi$ and estimate
	\begin{align}
	\no{N_{p,q}^{\delta,\nu}\varphi}^2_{s}\le c_s\left( \no{ \Box^{\delta,\nu}N_{p,q}^{\delta,\nu}\varphi}^2_s
	+\no{N_{p,q}^{\delta,\nu} \varphi}_0^2\right)
	&=c_s\left( \no{\varphi}^2_{s}+\no{N_{p,q}^{\delta,\nu} \varphi}_{0}^2\right)\nn \\
	&\le c_s \no{\varphi}^2_{s}+c_{s,\nu}\no{\varphi}^2_{0} \label{eqn:G delta nu s in terms of s and 0}
	\end{align}  
	for any positive integer $s$. The equality in \eqref{eqn:G delta nu s in terms of s and 0}
	follows from the identity $\Box^{\delta,\nu }N_{p,q}^{\delta,\nu}=Id$ (since $\ker(\Box^{\delta,\nu} )=\{0\}$), 
	and the inequality follows by \eqref{eqn:L2 G_q delta nu} and the fact that the constants $c_s, c_{s,\nu}$ are independent of $\delta>0$. 
	
	Thus, for each integer $s \geq 0$ and fixed $\nu>0$, $\no{N_{p,q}^{\delta,\nu} \varphi}_s$ is uniformly bounded in $\delta$, therefore, given any sequence $\delta_k \to 0$, there exists a subsequence
	$\delta_{k_j}$ and $u_{\nu,\{k_j\}} \in H^s_{0,q}(\Om)$ such that $N^{\delta_{k_j},\nu}_{p,q} \varphi\to  u_{\nu,\{k_j\}}$ weakly in $H^s_{p,q}(\Om)$. 
	Consequently, by a Cantor style diagonalization argument, we may find a $(p,q)$-form $u_\nu$ and a sequence of recursively defined subsequences 
	$\delta_{k,\ell}$ so that $N^{\delta_{k,\ell},\nu}_q \varphi\to u_\nu$ 
	weakly in $H^\ell_{p,q}(\Om)$ as $k\to\infty$ for every $\ell \in\mathbb{N}$. 
	Thus, it follows that $u_\nu \in C^\infty_{p,q}(\bar\Om)$ and if we redefine $\delta_k := \delta_{k,k}$
	then $N^{\delta_k,\nu}_{p,q} \varphi\to u_\nu$ weakly in $H^\ell_{p,q}(\Om)$ for all $\ell \geq 0$. 
	Additionally,
	$N^{\delta_k,\nu}_{p,q} \varphi\to  u_\nu$ weakly in the  $Q^{0,\nu}(\cdot,\cdot)^{1/2}$-norm. This
	means that if $v\in H^2_{p,q}(\Om)$, then 
	$$\lim_{k\to\infty}Q^{0,\nu}(N^{\delta_k,\nu}_{p,q}\varphi,v) 
	=Q^{0,\nu}(u_\nu,v).$$
	On the other hand, 
	$$Q^{0,\nu}(N^{0,\nu}_{p,q}\varphi,v)  =(\varphi,v)=
	Q^{\delta,\nu}(N^{\delta,\nu}_{p,q} \varphi,v)=Q^{0,\nu}(N^{\delta,\nu}_{p,q} \varphi,v) +\delta(TN^{\delta,\nu}_{p,q} \varphi,Tv)$$
	for all $v\in H^2_{p,q}(\Om)$. It follows that
	\[
	\big|Q^{0,\nu}\big((N^{\delta,\nu}_{p,q} -N^{0,\nu}_{p,q})\varphi,v\big)\big|
	\le \delta \no{N^{\delta,\nu}_{p,q}\varphi}_0 \no{v}_{2}\le \delta c_\nu\no{\varphi}_0\no{v}_2\to 0\quad \T{as $\delta\to 0$}
	\]
	where we have again used the inequality  $\no{N^{\delta,\nu}_{p,q} \varphi}_0\le c_\nu\no{\varphi}_0$ uniformly in $\delta\geq 0$. 
	Since $\ker \Box^{\delta,\nu}=\{0\}$, it follows that $N^{0,\nu}_{p,q}\vp =  u_\nu$ and hence $N^{0,\nu}_{p,q}\vp \in C^\infty_{p,q}(\bar\Om)$.
	Moreover, we may apply \eqref{eqn:Box delta nu} with $u=N^{0,\nu}_{p,q}\varphi \in C^\infty_{p,q}(\bar\Om)\cap \T{Dom}(\Box)$ and $\delta=0$ and observe
	\begin{equation}\label{est:G0nu}
	\no{N^{0,\nu}_{p,q}\varphi}^2_{s}\le c_s\left( \no{\varphi}^2_{s}+\no{N^{0,\nu}_{p,q} \varphi}^2_0\right),\end{equation}
	holds for all $\vp\in C^\infty_{p,q}(\bar\Om)$. Next, we show that harmonic forms are smooth.

{\bf Step 2:  $\H_{p,q}(\Om) \subset C^\infty_{p,q}(\bar\Om)$}

This step follows the ideas of \cite[Section 5]{Koh73}.
By Lemma~\ref{lem:lem below}, the $L^2$ harmonic space $\H_{p,q}(\Om)$  is finite dimensional. 
	Let $\theta_1,\cdots,\theta_N\in L^2_{p,q}(\Om)$ be a basis of $\H_{p,q}(\Om)$. 
	Set $\theta_0=0$. We will prove $\theta_j\in C_{p,q}^\infty(\bar\Om)$ for all $j$ by induction. Certainly $\theta_0 \in C_{p,q}^\infty(\bar\Om)$. Assume now that
	$\theta_j\in C_{p,q}^\infty(\bar\Om)$ for $0 \leq j \leq k < N$.  
	We will construct $\theta\in \H_{p,q}(\Om)\cap C_{p,q}^\infty(\bar\Om)$ with $\no{\theta}_{0}=1$ and $(\theta,\theta_j)=0$ for $j\le k$. In this way we obtain a basis of 
	$\H_{p,q}(\Om)$ which is contained in $C^\infty_{p,q}(\bar\Om)$. Let $\varphi\in C^\infty_{0,q}(\Om)$ such that $\varphi$ is orthogonal to $\theta_j$ for $j\le k$ but  not for $\theta_{k+1}$. 
	Then, for $\nu>0$, $N^{0,\nu}_{p,q}\varphi\in C^\infty_{p,q}(\bar\Om)\cap \T{Dom}(\Box)$ and satisfies \eqref{est:G0nu}.  
	We claim that $\{\no{N^{0,\nu}_{p,q}\varphi}_{0} : 0 < \nu < 1 \}$ is unbounded. If it were bounded then by \eqref{est:G0nu} we could find a subsequence converging to a form $u\in C^\infty_{p,q}(\bar\Om)\cap \T{Dom}(\Box)$ 
	satisfying  $$Q(u,\psi)=(\vp,\psi)$$
	for all $\psi\in L^2_{p,q}(\Om)$. Setting $\psi=\theta_j$, the left-hand side is zero for all $j$ and the right-hand side is different from zero for $j=k+1$, which is a contradiction. 
	Thus the sequence $\{\no{N_{p,q}^{0,\nu}\varphi}_{0}\}$ is unbounded and hence we can find a subsequence $\{\no{N_{p,q}^{0,\nu_m}\varphi}\}$ 
	such that $\lim_{m\to \infty}\no{N_{p,q}^{0,\nu_m}\varphi}_{0}=\infty$. 
	Set $w_{m}=\frac{N_{p,q}^{0,\nu_m}\varphi}{\no{N_{p,q}^{0,\nu_m}\varphi}_{0}}$. Then $w_m\in C^\infty_{p,q}(\bar\Om)\cap \T{Dom}(\Box)$, $\no{w_m}_{0}=1$, and by \eqref{est:G0nu}
	$$\no{w_{m}}_{H^s}\le c_s \left(\frac{\no{\vp}_{H^s}}{\no{N_{p,q}^{0,\nu_m}\varphi}_0}+1\right).$$ 
	Thus, there is a subsequence of $\{w_{m_i}\}$ such that $\lim_{i\to \infty}w_{m_i}=\theta\in C^\infty_{p,q}(\Om)$. The convergence occurs
	weakly in $H^\ell_{p,q}(\Om)$ and the compact inclusion $H^\ell_{p,q}(\Om) \hookrightarrow H^{\ell-1}_{p,q}(\Om)$ forces
	norm convergence of $w_{m_i}$ to $\theta$ in $H^{\ell-1}_{p,q}(\Om)$. Thus, $\no{\theta}_0=1$. To see that $\theta\in \H_{p,q}(\Om)$, we compute
	$$Q(w_{m_i},w_{m_i})\le Q^{0,\nu_{m_i}}(w_{m_i},w_{m_i})=\frac{1}{\no{N_{p,q}^{0,\nu_{m_i}}\varphi}_{0}}(\vp, w_{m_i})\le \frac{\no{\varphi}_{0}}{\no{N_{p,q}^{0,\nu_{m_i}}\varphi}_{0}},$$
	send $i\to\infty$, and use the fact that $w_{m_i} \to \theta$ in $H^1_{p,q}(\Om)$ to
	conclude that $Q(\theta,\theta)=0$. In other words, $\theta\in \H_{p,q}(\Om)$. Finally, to prove $(\theta,\theta_j)=0$ for $j\le k$, we set $\psi=\theta_j$ and observe
	$$\nu_m(w_m,\psi)=Q^{\nu_m}(w_m,\psi)=\frac{1}{\no{N_{p,q}^{0,\nu_{m}}\varphi}_0}(\vp,\psi)=0.$$
	It follows that $w_m$ and hence $\theta$ is orthogonal to $\theta_k$ for $j=1,\dots,k$. Therefore, $\H_{p,q}(\Om)\subset C^\infty_{p,q}(\bar\Om)$.\\
	
	{\bf Step 3: $N_{p,q}:=N^{0,0}_{p,q}$ is both globally and exactly regular.}
	
We start this step by using Lemma~\ref{lem:lem below} and \eqref{eqn:L2 appendix} to observe
	$$\no{u}_0^2\le c\left(\no{\dib u}_0^2+\no{\dib^* u}_0^2\right)=c Q(u,u)$$
 and hence
	\begin{equation}\label{eqn:uniformly}
	\no{u}_0^2\le c \left( Q(u,u)+\nu \no{u}_0^2 \right)=c Q^{0,\nu}(u,u).
	\end{equation}
	for any $u\in \T{Dom}(\dib)\cap \T{Dom}(\dib^*)\cap \H^{\perp}_{(p,q}(\Om)$, where $c$ is independent of $\nu$.	
	
By the definition of $Q^{0,\nu}$ and $N_{p,q}^{0,\nu}$, we have
	$$\nu(N_{p,q}^{0,\nu}\varphi, \psi) = Q(N_{p,q}^{0,\nu}\varphi, \psi)+\nu(N_{p,q}^{0,\nu}\varphi, \psi)=Q^{0,\nu}(N_{p,q}^{0,\nu}\varphi, \psi)=(\varphi,\psi)$$
	for any $\varphi,\psi\in L^2_{p,q}(\Om)$. This calculation shows that $N_{p,q}^{0,\nu}\varphi\perp  \H_{p,q}(\Om)$ whenever $\vp\perp \H_{p,q}(\Om)$ because $Q(f,\psi) =0$ 
	 for all $\psi \in  \H_{p,q}(\Om)$ and $f \in \Dom \dbar \cap \Dom\dbars$.
Thus, if $u=N_{p,q}^{0,\nu}\varphi$ and $\varphi\perp\H_{p,q}(\Om)$, then the uniformity of \eqref{eqn:uniformly} (in $\nu >0$) implies
	$$\no{N_{p,q}^{0,\nu}\varphi}_{0}\le c\no{\varphi}_{0}.$$ 
	Combining this uniform $L^2$ estimate with \eqref{est:G0nu} yields the uniform (in $\nu>0$) $H^s$ estimate
	\begin{equation}\label{est:G0nu phi}
	\no{N_{p,q}^{0,\nu}\varphi}^2_s\le c_s\left( \no{\varphi}^2_{s}+\no{\varphi}^2_{0}\right)\le c_s\no{\varphi}^2_{s},\end{equation}
	for any $\varphi\in C^\infty_{p,q}(\bar\Om)\cap \H^{\perp}_{p,q}(\Om)$.
	Now we use the same argument as in Step 1  to  show that $N_{p,q}\vp\in C^\infty_{p,q}(\bar\Om)\cap \H^{\perp}_{p,q}(\Om)$ and \eqref{est:G0nu phi} holds for $\nu=0$. For $\vp\in C^\infty_{p,q}(\bar\Om) $, we decompose 
	$\vp=(I- H_{p,q})\vp+ H_{p,q}\vp$. 
	Since $\vp\in C^\infty_{p,q}(\bar\Om)$, it follows from Step 2 that $(I- H_{(p,q})\vp\in C^\infty_{p,q}(\bar\Om)$, and by using \eqref{est:G0nu phi} for $\nu=0$, we may conclude that 
	\begin{equation}\label{est:concludeA}
	\no{N_{p,q}\varphi}^2_{s}\le c_s \no{(I- H_{p,q}) \varphi}^2_{s}\le c_s\no{\varphi}^2_{s} ,\end{equation}
	for all $\vp\in C^\infty_{p,q}(\bar\Om)$. Hence $N_{p,q}$ is globally regular and by the density of $C^\infty_{p,q}(\bar\Om)$ in $H^s_{p,q}(\Om)$, it follows that
	\eqref{est:concludeA} holds for any $\vp\in H^s_{p,q}(\Om)$. Hence $N_{p,q }$ is exactly regular as well.

\end{proof}

\bibliographystyle{alpha}
\bibliography{mybib8-1-24}
\end{document}